\documentclass[11pt,a4paper]{article}
\usepackage[left=22mm,right=22mm,top=20mm,bottom=22mm]{geometry}
\usepackage{amsmath,amsthm,amssymb,mathtools}

\usepackage{microtype,booktabs,longtable,array,needspace}
\usepackage{xurl}
\usepackage[unicode,hidelinks]{hyperref}
\usepackage{bookmark}
\newtheorem{maintheorem}{Theorem}
\newtheorem{lemma}{Lemma}[section]
\newtheorem{proposition}[lemma]{Proposition}
\newtheorem{corollary}[lemma]{Corollary}
\theoremstyle{definition}

\newtheorem{definition}[lemma]{Definition}
\hypersetup{pdftitle={Hausdorff dimension of canonical height spectra},pdfauthor={Wayne Peng}}
\title{Hausdorff dimension of canonical height spectra}
\author{Wayne Peng\\[3pt]
\small Department of Mathematics, National Central University\\
\small\href{mailto:Junwen.Wayne.peng@math.ncu.edu.tw}{\texttt{Junwen.Wayne.peng@math.ncu.edu.tw}}}
\date{}
\begin{document}
\maketitle
\begin{abstract}
We study the Hausdorff dimension of the closure of the canonical
height values associated to finite compositions of two rational maps.
For rational maps with common periodic critical points, a nonzero
local--global height discrepancy gives a nonempty set of starting
points with positive-dimensional spectra. For polynomials, equality
of the normalized logarithms of the leading coefficients at every place
forces dimension zero; a single-place discrepancy gives positive
dimension under boundedness at the other places. We construct rational
cubic pairs with spectra of dimension \(\log2/\log3\).
The spectrum dimension gives a lower bound for the growth exponent of
the number of distinct orbit points of bounded height. At the degree
exponent, positive Hausdorff measure gives matching growth bounds for
multiplicatively dependent degrees. For independent degrees, a
restriction of the covering intervals and a Diophantine rotation estimate
give a lower bound with logarithmic loss of exponent strictly less than
\(1/2\).
Quadratic monomial examples show that the spectrum dimension can depend
on the starting point even when both individual orbits are infinite.
\end{abstract}

\section{Introduction}\label{sec:introduction}
Let $\phi,\psi\in\overline{\mathbb Q}(t)$ be rational maps of degree at least $2$. Each nonempty finite composition $u$ of these maps has a dynamical canonical height $\widehat h_u$. For $P\in\mathbb P^1(\overline{\mathbb Q})$, we consider
\[
X_{\phi,\psi}(P)=\{\widehat h_u(P):u\in\langle\phi,\psi\rangle,\ u\ne\operatorname{id}\},
\qquad
\widehat X_{\phi,\psi}(P)=\overline{X_{\phi,\psi}(P)}\subseteq\mathbb R.
\]
The identity is allowed as an empty prefix, but is excluded from the height indices. We call $\widehat X_{\phi,\psi}(P)$ the completed height spectrum. The set of finite compositions is countable, so $X_{\phi,\psi}(P)$ itself always has Hausdorff dimension zero. The closure can be much larger.

Canonical heights describe the growth of Weil heights under iteration; see Call--Silverman~\cite{CallSilverman1993} and Silverman~\cite{Silverman2007}. Kawaguchi~\cite{Kawaguchi2007} constructed corresponding heights for bounded sequences of morphisms. Related height-growth questions for nonautonomous iteration are studied by Hindes~\cite{Hindes2021Growth}. In our finite-generator setting, the sequence-height map is continuous, and its image is the completed height spectrum. Thus the spectrum records all height limits obtained by allowing the generators to vary along an infinite sequence.

The degrees give a general upper bound. If $s_0>0$ satisfies
\[
d_\phi^{-s_0}+d_\psi^{-s_0}=1,
\]
then $\dim_H\widehat X_{\phi,\psi}(P)\le s_0\le1$; see Proposition~\ref{prop:hausdorff-upper-bound}. To obtain a lower bound, we use separation between canonical heights along the composition orbit. The height estimate underlying this argument is related to the height ping-pong method in Bell--Huang--Peng--Tucker~\cite[Proposition~3.1 and Lemma~3.2]{BHPT2024}. Here the separation yields a quantitative Hausdorff-dimension bound.

Beaumont~\cite[Sections~2--3]{Beaumont2025UniformTits} develops
the contraction argument on a space of height functions to prove
a uniform Tits alternative. The spectra considered here are
evaluation images at a fixed algebraic point. We study their
Hausdorff dimension through local arithmetic criteria and relate
their Hausdorff measure to orbit counts by height.

We first give a criterion for positive-dimensional spectra using
common periodic critical points. Different common points may be
selected at different places.

\begin{maintheorem}\label{thm:main-1}\label{prop:critical-basin-criterion}
Let \(K\) be a number field, let \(\phi,\psi\in K(t)\) have degrees at least \(2\), and fix homogeneous lifts \(\widetilde\phi,\widetilde\psi\) over \(K\), for use at every place. Let \(S\subseteq M_K\) be finite, containing all archimedean places, such that both lifts have integral coefficients and unit resultant at every finite place outside \(S\).

For each \(v\in S\), choose a point \(\alpha_v\in\mathbb P^1(K)\) that is periodic and critical for both maps. Choose \(r,s\ge1\) such that
\[
\phi^r(\alpha_v)=\psi^s(\alpha_v)=\alpha_v\qquad(v\in S).
\]
These return times are chosen to work for every \(v\in S\), by taking
common multiples of the finitely many periods.
For a nonzero representative \(\widetilde\alpha_v\), write
\[
\widetilde\phi^{ r}(\widetilde\alpha_v)
=\lambda_{\phi,v}\widetilde\alpha_v,
\qquad
\widetilde\psi^{ s}(\widetilde\alpha_v)
=\lambda_{\psi,v}\widetilde\alpha_v.
\]
Set
\[
\Lambda=
\frac1{[K:\mathbb Q]}
\sum_{v\in S}n_v
\left(
\frac{\log|\lambda_{\phi,v}|_v}{d_\phi^r-1}
-
\frac{\log|\lambda_{\psi,v}|_v}{d_\psi^s-1}
\right).
\]
If \(\Lambda\ne0\), there are closed coordinate disks \(D_v\) centered at \(\alpha_v\), preserved by both \(\phi^r\) and \(\psi^s\), such that
\[
\mathcal U=
\{P\in\mathbb P^1(K):P\in D_v\text{ for every }v\in S\}
\]
is nonempty and
\[
\dim_H\widehat X_{\phi,\psi}(P)>0
\]
for $P\in\mathcal U$.
\end{maintheorem}

The nonzero condition on \(\Lambda\) prevents cancellation between
the local contributions. The proof, given in
Section~\ref{sec:general-criteria}, combines local canonical-height
estimates with Lemma~\ref{lem:orbitwise-gap}.

For polynomials, the local behavior near infinity provides concrete criteria. The normalized logarithms of the leading coefficients distinguish a zero-dimensional case from a sufficient condition for positive dimension.

\begin{maintheorem}\label{thm:main-2}\label{prop:polynomial-zero}\label{prop:polynomial-positive}
Let \(\phi,\psi\in K[t]\) have degrees \(d_\phi,d_\psi\ge2\)
and leading coefficients \(a_\phi,a_\psi\), where \(K\) is a number
field. For \(v\in M_K\), put
\[
\delta_v=\frac{\log|a_\phi|_v}{d_\phi-1}
          -\frac{\log|a_\psi|_v}{d_\psi-1}.
\]
\begin{enumerate}
\renewcommand{\theenumi}{\roman{enumi}}
\renewcommand{\labelenumi}{(\theenumi)}
\item\label{part:main-2-zero}
If
\begin{equation}\label{eq:zero-1}
\delta_v=0\qquad(v\in M_K),
\end{equation}
then
\[
\dim_H\widehat X_{\phi,\psi}(z)=0
\qquad(z\in\overline{\mathbb Q}).
\]
In particular, this holds when both polynomials are monic.

\item\label{part:main-2-positive}
If \(\delta_v\ne0\) at a place \(v\), there exists \(R>0\),
depending only on \(\phi,\psi\) and \(v\), such that
\[
\dim_H\widehat X_{\phi,\psi}(z)>0
\]
for every \(z\in K\) satisfying \(|z|_v>R\) and
\[
\sup_{y\in\mathcal O(z)}|y|_{v'}<\infty
\qquad(v'\in M_K\setminus\{v\}),
\]
where \(\mathcal O(z)\) is the full composition orbit, including
\(z\).
\end{enumerate}
\end{maintheorem}

The boundedness hypothesis in part~\textup{(ii)} concerns the entire composition orbit at every other place. Both assertions are proved in Section~\ref{sec:applications}. Combining part~\textup{(ii)} with the degree upper bound gives intermediate dimensions for polynomial pairs satisfying $d_\phi^{-1}+d_\psi^{-1}<1$; see Corollary~\ref{cor:intermediate-dimension}.

Homogeneous local sequence heights provide the local--global comparison used in Theorem~\ref{thm:main-1}; see Young~\cite[Propositions~4.6 and 4.12]{YoungThesis2023}. In Section~\ref{subsec:rational-cases}, we construct cubic pairs that are not M\"obius-conjugate to polynomials and have completed spectra of dimension $\log2/\log3$, with positive finite Hausdorff measure in that dimension.

The dimension also has an arithmetic application. Orbit counts by height for semigroups were studied by Hindes~\cite{Hindes2022Monoid} and by Bell--Hindes--Zhong~\cite{BellHindesZhongApplications}. Their general upper bound for distinct orbit points can be compared with the following lower bound in terms of the height spectrum.

\begin{maintheorem}\label{thm:main-3}
Let \(K\) be a number field, let \(\phi,\psi\in K(t)\) have
 degrees \(d_\phi,d_\psi\ge2\), and let \(P\in\mathbb P^1(K)\).
Let \(s_0>0\) satisfy
\[
d_\phi^{-s_0}+d_\psi^{-s_0}=1,
\]
and put
\[
N_P(X)=\#\{Q\in\mathcal O(P):h(Q)\le X\}.
\]
Suppose that
\[
\mathcal H^{s_0}\bigl(\widehat X_{\phi,\psi}(P)\bigr)>0.
\]
\begin{enumerate}
\renewcommand{\theenumi}{\roman{enumi}}
\renewcommand{\labelenumi}{(\theenumi)}
\item\label{part:main-3-dependent}
If \(d_\phi,d_\psi\) are multiplicatively dependent, then
\[
cX^{s_0}\le N_P(X)\le CX^{s_0}
\]
for all sufficiently large \(X\).
\item\label{part:main-3-independent}
If \(d_\phi,d_\psi\) are multiplicatively independent, there
exists \(\eta>0\), depending only on the degrees, such that
\[
c\,\frac{X^{s_0}}{(\log X)^{1/2-\eta}}
\le N_P(X)\le CX^{s_0}
\]
for all sufficiently large \(X\).
\end{enumerate}
The positive constants \(c,C\) may depend on \(\phi,\psi,P\).
\end{maintheorem}

For arbitrary positive spectrum dimension \(\delta\),
Proposition~\ref{prop:orbit-counting-lower} gives a lower
orbit-counting exponent of at least \(\delta\). Theorem~\ref{thm:main-3}
uses positive Hausdorff measure at the degree exponent \(s_0\)
to give sharper bounds. In the independent-degree case, the proof
restricts the covering intervals to a short range of letter counts
and bounds, uniformly in the height, the number of remaining
degrees through an irrational rotation.

No freeness or injectivity assumption on word evaluations is required in this theorem. Section~\ref{sec:orbit-counting} gives the proof and matching counting bounds for the cubic family. Section~\ref{sec:discussion} discusses quadratic examples, dependence on the starting point, and further applications and questions.

\section{Preliminaries}\label{sec:preliminaries}

\subsection{Heights}
We first recall the height, dynamical canonical height, and the estimates used below. Details can be found in Silverman's textbook~\cite{Silverman2007}.

Let \(K\) be a number field and \(M_K\) its set of places. Normalize \(|\cdot|_v\) to extend the standard absolute values on \(\mathbb Q\): thus \(|p|_v=p^{-1}\) for \(v\mid p\), while archimedean places use the usual absolute value or complex modulus. Put
\[
n_v=[K_v:\mathbb Q_{v_0}],
\]
where \(v_0\) is the place of \(\mathbb Q\) below \(v\).

For \(P=[a:b]\in\mathbb P^1(\overline{\mathbb Q})\), choose such a field \(K\) containing \(a,b\). The \emph{absolute logarithmic Weil height} of \(P\) is
\[
h(P)=\frac{1}{[K:\mathbb Q]}
\sum_{v\in M_K}
n_v\log\max\{|a|_v,|b|_v\}.
\]
The product formula and the local-degree identity imply that this definition is independent of the homogeneous coordinates and the choice of \(K\).
For polynomial maps we use affine coordinates, with \(h(z)=h([z:1])\); for general rational maps we use projective coordinates.

For \(\phi\in\overline{\mathbb Q}(t)\) of degree \(d\ge2\), its dynamical canonical height is
\[
\widehat h_\phi(P)=\lim_{n\to\infty}\frac{h(\phi^n(P))}{d^n},
\qquad P\in\mathbb P^1(\overline{\mathbb Q}).
\]
We will use the following estimate. There exists a constant $C_\phi$, depending only on $\phi$, such that
\begin{equation}\label{eq:classical-bound}
\left|\widehat h_\phi(P)-\frac{h(\phi^n(P))}{d^n}\right|
\le\frac{C_\phi}{(d-1)d^n}
\end{equation}
for all $n\ge 0$.

\Needspace{7\baselineskip}
\subsection{Height spectra and cylinders}

\Needspace{6\baselineskip}
\begin{definition}\label{def:height-spectrum}
    Let $\phi,\psi\in\overline{\mathbb Q}(t)$ have degree at least $2$.  Write
    \[
        \langle\phi,\psi\rangle
        =
        \left\{\theta_n\circ\cdots\circ\theta_1 : n\ge 1,\ \theta_i\in\{\phi,\psi\}\right\}\cup\{\operatorname{id}\}
    \]
    for the semigroup generated by $\phi$ and $\psi$ under composition. For a composition $u=\theta_n\circ\cdots\circ\theta_1\in\langle\phi,\psi\rangle$, define its length by $|u|\coloneqq n$, with $|\operatorname{id}|=0$. Since different words may induce the same map, $|u|$ is a property of the chosen composition, not of the induced map itself.

    For $P\in\mathbb P^1(\overline{\mathbb Q})$, define the \emph{height spectrum} of $P$ by
    \[
        X_{\phi,\psi}(P)\coloneqq\{\widehat h_u(P): u\in\langle\phi,\psi\rangle,\ u\ne\operatorname{id}\},
    \]
    and the \emph{completed height spectrum} by
    \[
        \widehat X_{\phi,\psi}(P)\coloneqq\overline{X_{\phi,\psi}(P)}\subseteq\mathbb R,
    \]
    where the closure is taken in the usual Euclidean topology on $\mathbb R$.  When $\phi$ and $\psi$ are clear from context, we write $X(P)$ and $\widehat X(P)$.
\end{definition}
\Needspace{10\baselineskip}
The \emph{sequence height} defined below (see \cite{Kawaguchi2007}) is obtained as a limit of dynamical canonical heights associated to finite prefixes of $\omega$. In particular, it belongs to the completed spectrum of Definition~\ref{def:height-spectrum}.

\begin{definition}\label{def:sequence-height}
    Let $\omega=(\theta_1,\theta_2,\ldots)\in\{\phi,\psi\}^{\mathbb N}$. Define
    \[
    \widehat h_\omega(P)\coloneqq
    \lim_{n\to\infty}
    \frac{h\bigl((\theta_n\circ\cdots\circ\theta_1)(P)\bigr)}{\prod_{i=1}^{n} d_{\theta_i}}.
    \]
    Here $\theta_1$ is applied first. The existence of this limit follows from the telescoping estimate in Lemma~\ref{lem:sequence-bound} below.

    For $u=u_n\circ\cdots\circ u_1\in\langle\phi,\psi\rangle$, define the concatenated sequence
    \[
    \omega\circ u \coloneqq (u_1,\ldots,u_n,\theta_1,\theta_2,\ldots).
    \]
\end{definition}
We shall prove an inequality analogous to Equation~\eqref{eq:classical-bound}. From now on we fix \(\phi,\psi\in\overline{\mathbb Q}(t)\) with \(d_\phi,d_\psi\ge2\), and fix constants \(C_\phi,C_\psi\ge0\) such that
\[
|h(\theta(Q)) - d_\theta\, h(Q)|\le C_\theta
\qquad
\bigl(\theta\in\{\phi,\psi\},\ Q\in\mathbb P^1(\overline{\mathbb Q})\bigr).
\]
Throughout this section, set
\[
B\coloneqq \max\left\{\frac{C_\phi}{d_\phi-1},\ \frac{C_\psi}{d_\psi-1}\right\}.
\]

\begin{lemma}\label{lem:sequence-bound}
    Let $\omega=(\theta_1,\theta_2,\ldots)\in\{\phi,\psi\}^{\mathbb N}$, and write $\omega_n\coloneqq \theta_n\circ\cdots\circ\theta_1$, with $\omega_0\coloneqq \operatorname{id}$ (so $\deg(\omega_0)=1$). Then for any $P\in\mathbb P^1(\overline{\mathbb Q})$ and any $n\ge 0$,
    \begin{equation}\label{eq:sequence-bound}
    \left|\widehat h_\omega(P)-\frac{h(\omega_n(P))}{\deg \omega_n}\right|\le \frac{B}{\deg\omega_n}.
    \end{equation}
\end{lemma}
\begin{proof}
    For $k\ge 0$, we have
    \[
    \begin{aligned}
        \left|
        \frac{h(\omega_k(P))}{\deg\omega_k}
        -\frac{h(\omega_{k+1}(P))}{\deg\omega_{k+1}}
        \right|
        &\le \frac{C_{\theta_{k+1}}}{\deg\omega_{k+1}}\\
        &\le \frac{B(d_{\theta_{k+1}}-1)}{\deg\omega_{k+1}}\\
        &=B\left(\frac1{\deg\omega_k}-\frac1{\deg\omega_{k+1}}\right).
    \end{aligned}
    \]
    By telescoping and passing to the limit, we obtain
    \[
    \left|\widehat h_\omega(P)-\frac{h(\omega_n(P))}{\deg \omega_n}\right|
    \le \sum_{k=n}^\infty \left|\frac{h(\omega_k(P))}{\deg\omega_k}-\frac{h(\omega_{k+1}(P))}{\deg\omega_{k+1}}\right|
    \le B\sum_{k=n}^\infty\left(\frac{1}{\deg \omega_k}-\frac{1}{\deg\omega_{k+1}}\right),
    \]
    which yields the desired result.
\end{proof}
The next estimate plays a key role in the elementary ping-pong argument for the dynamical Tits alternative; see \cite[Proposition~3.1 and Lemma~3.2]{BHPT2024}.

\Needspace{6\baselineskip}
\begin{lemma}\label{lem:height-ping-pong}
For any \(\omega,\omega'\in\{\phi,\psi\}^{\mathbb N}\), any $u\in\langle\phi,\psi\rangle$, and any \(P\in\mathbb P^1(\overline{\mathbb Q})\),
\[
\left|\widehat h_{\omega\circ u}(P)-\widehat h_{\omega'\circ u}(P)\right|
\le \frac{2B}{\deg u}.
\]

As a consequence, if we set
\[
\Delta(P)\coloneqq \left|\widehat h_\phi(P)-\widehat h_\psi(P)\right|,
\]
then for all sequences \(\omega,\omega'\) and all \(m,n\ge 1\) we obtain
\[
\left|\widehat h_{\omega\circ\phi^m}(P)-\widehat h_{\omega'\circ\psi^n}(P)\right|
\ge \Delta(P)-2B\bigl(d_\phi^{-m}+d_\psi^{-n}\bigr).
\]
In particular, these two sequence heights differ whenever
\[
2B\bigl(d_\phi^{-m}+d_\psi^{-n}\bigr)<\Delta(P).
\]
\end{lemma}
\begin{proof}
The first inequality follows from the triangle inequality and Lemma~\ref{lem:sequence-bound}. For the second inequality, we have
\begin{align*}
\left|\widehat h_{\omega\circ\phi^m}(P)-\widehat h_{\omega'\circ\psi^n}(P)\right|
&=\left|\widehat h_{\omega\circ\phi^m}(P)-\widehat h_{\phi}(P)+\widehat h_{\phi}(P)-\widehat h_{\psi}(P)+\widehat h_{\psi}(P)-\widehat h_{\omega'\circ\psi^n}(P)\right|\\
&\ge \left|\widehat h_{\phi}(P)-\widehat h_{\psi}(P)\right| 
- \left|\widehat h_{\omega\circ\phi^m}(P)-\widehat h_{\phi}(P)+\widehat h_{\psi}(P)-\widehat h_{\omega'\circ\psi^n}(P)\right|\\
&\ge \left|\widehat h_{\phi}(P)-\widehat h_{\psi}(P)\right| 
- \left|\widehat h_{\omega\circ\phi^m}(P)-\widehat h_{\phi}(P)\right| 
- \left|\widehat h_{\psi}(P)-\widehat h_{\omega'\circ\psi^n}(P)\right|\\
&\ge \Delta(P)-\left(\frac{2B}{d_\phi^m}+\frac{2B}{d_\psi^n}\right).
\end{align*}
\end{proof}

\begin{definition}\label{def:height-cylinders}
Let \(\phi,\psi\in\overline{\mathbb Q}(t)\) have degrees at least \(2\), and let \(P\in\mathbb P^1(\overline{\mathbb Q})\). For $u\in\langle\phi,\psi\rangle$, the \emph{height cylinder with prefix \(u\)} is
\[
X_{\phi,\psi}^{u}(P)
\coloneqq
\left\{\widehat h_{u'\circ u}(P) : u'\in\langle\phi,\psi\rangle,\ u'\circ u\ne\operatorname{id}\right\}.
\]

Its \emph{completion} is
\[
\widehat X_{\phi,\psi}^{u}(P)
\coloneqq \overline{X_{\phi,\psi}^{u}(P)},
\]
where the closure is taken in \(\mathbb R\) with its usual metric. When the maps are fixed, write \(X^u(P)\) and \(\widehat X^u(P)\).
\end{definition}

\subsection{Local sequence height}\label{subsec:local-sequence-height}
Let $K$ be a number field over which $\phi$ and $\psi$ are defined. A \emph{homogeneous lift} of a rational map $\theta\in K(t)$ of degree $d_\theta\ge2$ is a morphism $\widetilde\theta:\mathbb A^2\to\mathbb A^2$ whose coordinate polynomials lie in $K[X,Y]$, are homogeneous of degree $d_\theta$, have no common zero outside $(0,0)$, and represent $\theta$. Throughout the paper, fix homogeneous lifts $\widetilde\phi,\widetilde\psi$ over $K$ and use these same lifts at every place.

For a finite word $u=\theta_n\circ\cdots\circ\theta_1$, we use the composed lift
\[
\widetilde u=\widetilde\theta_n\circ\cdots\circ\widetilde\theta_1,
\]
without further normalization.

We now define local sequence heights, including local dynamical canonical heights as the periodic case; see \cite[Proposition~4.6]{YoungThesis2023} for the construction.
\begin{definition}\label{def:local-sequence-height}
For $\omega=(\theta_1,\theta_2,\ldots)\in\{\phi,\psi\}^{\mathbb N}$ and a place $v\in M_K$, write
\[
\|(a,b)\|_v=\max\{|a|_v,|b|_v\}.
\]
The \emph{local sequence height} associated to the fixed lifts is
\[
\widehat h_{\omega,v}(\widetilde P)
=
\lim_{n\to\infty}
\frac{\log\|\widetilde{\omega_n}(\widetilde P)\|_v}{\deg\omega_n},
\qquad \widetilde P\in K_v^2\setminus\{(0,0)\}.
\]
The limit exists by a uniform telescoping estimate for the two fixed lifts.
If $\omega$ is obtained by repeating a nonempty finite block $(\theta_1,\ldots,\theta_n)$, put $u=\theta_n\circ\cdots\circ\theta_1$ and write
\[
\widehat h_{u,v}\coloneqq\widehat h_{\omega,v}.
\]
This is the local dynamical canonical height associated to the induced lift $\widetilde u$. The local notation refers to the chosen word and its induced lift.
\end{definition}

The choice of lifts matters locally. For any $c\in K_v^\times$, we have
\begin{equation}\label{eq:scaling-identity}
\widehat h_{\omega,v}(c\widetilde P)
=\widehat h_{\omega,v}(\widetilde P)+\log|c|_v.
\end{equation}

For $P=[a:b]\in\mathbb P^1(K)$, the local--global decomposition is
\begin{equation}\label{eq:global-local-decomposition}
\widehat h_\omega(P)
=
\frac1{[K:\mathbb Q]}
\sum_{v\in M_K}n_v\widehat h_{\omega,v}(a,b).
\end{equation}
Only finitely many places contribute, uniformly in the sequence: outside a fixed finite set, the lifts have integral coefficients and unit resultants, and $\|(a,b)\|_v=1$. Applying the ordinary-height identity to each truncation and passing to the limit in this finite sum gives the decomposition. The product formula makes the sum independent of the homogeneous coordinates and the chosen global lifts; see also \cite[Proposition~4.12]{YoungThesis2023}.

For polynomial maps, choose
\[
\widetilde\theta(X,Y)
=
\bigl(Y^{d_\theta}\theta(X/Y),\,Y^{d_\theta}\bigr).
\]
Then
\[
\widetilde{\omega_n}(z,1)=(\omega_n(z),1),
\]
so, writing $\widehat h_{\omega,v}(z)\coloneqq\widehat h_{\omega,v}(z,1)$, we recover
\[
\widehat h_{\omega,v}(z)
=
\lim_{n\to\infty}
\frac{\log\max\{1,|\omega_n(z)|_v\}}{\deg\omega_n}.
\]
\subsection{Hausdorff dimension}
Let \((X,\rho)\) be a metric space and \(E\subseteq X\). For a nonempty subset \(U\subseteq X\), write
\[
\operatorname{diam} U
=\sup\{\rho(x,y):x,y\in U\}.
\]
For \(s>0\) and \(\delta>0\), define
\[
\mathcal H^s_\delta(E)
=\inf\left\{
\sum_i(\operatorname{diam} U_i)^s:
E\subseteq\bigcup_i U_i,\quad
\operatorname{diam} U_i\le\delta
\right\},
\]
where the infimum is taken over all finite or countable covers of \(E\) by nonempty subsets of \(X\).

The \emph{\(s\)-dimensional Hausdorff outer measure} of \(E\) is
\[
\mathcal H^s(E)
=\lim_{\delta\to0^+}\mathcal H^s_\delta(E).
\]
The limit exists in \([0,\infty]\), since \(\mathcal H^s_\delta(E)\) is nondecreasing as \(\delta\) decreases. The \emph{Hausdorff dimension} of \(E\) is
\[
\dim_H E=\inf\{s>0:\mathcal H^s(E)=0\},
\]
with the convention \(\inf\varnothing=\infty\).

\Needspace{6\baselineskip}
\begin{lemma}\label{lem:cylinder-containment}
For every \(P\in\mathbb P^1(\overline{\mathbb Q})\) and every $u\in\langle\phi,\psi\rangle$, we have
\[
\widehat X_{\phi,\psi}^{u}(P)\subseteq I_u(P)
\coloneqq
\left[
\frac{h(u(P))-B}{\deg u},
\frac{h(u(P))+B}{\deg u}
\right].
\]
In particular,
\[
\operatorname{diam}\widehat X_{\phi,\psi}^{u}(P)
\le \frac{2B}{\deg u}.
\]
\end{lemma}
\begin{proof}
By Lemma~\ref{lem:sequence-bound} and Equation~\eqref{eq:classical-bound}, the same constant \(B\) applies to both bounds below.

Take \(\xi\in X_{\phi,\psi}^{u}(P)\), represented by
\[
\xi=\widehat h_{u'\circ u}(P),\qquad
\text{with }u'\in\langle \phi,\psi\rangle.
\]
Equation~\eqref{eq:sequence-bound} gives
\[
\left|\xi-\frac{h(u'\circ u(P))}{\deg(u'\circ u)}\right|\le\frac {B}{\deg(u'\circ u)}.
\]
The finite telescoping estimate in the lemma gives
\[
\left|
\frac{h(u'\circ u(P))}{\deg(u'\circ u)}-\frac{h(u(P))}{\deg u}
\right|
\le B\left(\frac1{\deg u}-\frac1{\deg(u'\circ u)}\right).
\]
Adding these inequalities, we obtain
\[
\left|\xi-\frac{h(u(P))}{\deg u}\right|
\le\frac B{\deg (u'\circ u)}+B\left(\frac1{\deg u}-\frac1{\deg(u'\circ u)}\right)
=\frac B{\deg(u)}.
\]
Thus \(X^{u}(P)\subseteq I_u(P)\). Since \(I_u(P)\) is closed, the same holds for \(\widehat X^{u}(P)\).
\end{proof}

\Needspace{6\baselineskip}
\begin{proposition}\label{prop:hausdorff-upper-bound}
Let \(\phi,\psi\in\overline{\mathbb Q}(t)\) have degrees \(d_\phi,d_\psi\ge2\), and let \(s_0>0\) be the unique solution of
\[
d_\phi^{-s_0}+d_\psi^{-s_0}=1.
\]
Then, for every \(P\in\mathbb P^1(\overline{\mathbb Q})\),
\[
\dim_H\widehat X_{\phi,\psi}(P)\le s_0.
\]
In particular, if \(d_\phi=d_\psi=d\), then
\[
\dim_H\widehat X_{\phi,\psi}(P)\le\frac{\log2}{\log d}.
\]
\end{proposition}
\begin{proof}
For every \(n\ge1\),
\[
\widehat X_{\phi,\psi}(P)
=\bigcup_{|u|=n}\widehat X_{\phi,\psi}^{u}(P).
\]
Indeed, replacing a composition by a sufficiently high iterate leaves its dynamical canonical height unchanged and gives a word of length at least \(n\). Thus the cylinders of length \(n\) cover \(X(P)\); taking closures gives the equality above because the union is finite.

By Lemma~\ref{lem:cylinder-containment}, the completed cylinders satisfy
\[
\operatorname{diam}\widehat X^{u}(P)
\le\frac{2B}{\deg u}
\le 2B\,2^{-n}\longrightarrow0.
\]
Fix \(s>s_0\). Summing over all compositions $u=u_n\circ\cdots\circ u_1$ of length \(n\), we obtain
\[
\begin{aligned}
\sum_{|u|=n}
\left(\operatorname{diam}\widehat X^{u}(P)\right)^s
&\le (2B)^s\sum_{|u|=n}(\deg u)^{-s}\\
&\le (2B)^s\left(d_\phi^{-s}+d_\psi^{-s}\right)^n.
\end{aligned}
\]
The last bound follows by expanding the product over the two generator choices at each position. Repeated induced maps are counted with multiplicity, which can only increase the bound. Since \(s>s_0\),
\[
d_\phi^{-s}+d_\psi^{-s}
<d_\phi^{-s_0}+d_\psi^{-s_0}=1,
\]
so the covering sum tends to zero as \(n\to\infty\). Hence
\[
\mathcal H^s\!\left(\widehat X_{\phi,\psi}(P)\right)=0,
\]
and therefore \(\dim_H\widehat X(P)\le s_0\). The equal-degree formula follows from \(2d^{-s_0}=1\).
\end{proof}

\Needspace{7\baselineskip}
\begin{lemma}\label{lem:sequence-no-isolated-points}
Let \(\phi,\psi\in\overline{\mathbb Q}(t)\) be distinct rational maps of degrees at least \(2\), and let \(P\in\mathbb P^1(\overline{\mathbb Q})\). Suppose that distinct infinite sequences in \(\{\phi,\psi\}^{\mathbb N}\) beginning with \(\phi\) have distinct sequence heights at \(P\). Then every open interval containing one of these heights contains at least two distinct such heights.
\end{lemma}
\begin{proof}
Fix such a sequence \(\omega\). For each \(n\ge1\), change only its \((n+1)\)-st entry to the other map, obtaining a sequence \(\omega'\). The first entry remains \(\phi\), and the two heights are distinct by hypothesis.

Let \(u\) be the common prefix of length \(n\). Then \(\widehat h_\omega(P),\widehat h_{\omega'}(P)\in\widehat X^u_{\phi,\psi}(P)\). Lemma~\ref{lem:cylinder-containment} therefore gives
\[
0<\left|\widehat h_\omega(P)-\widehat h_{\omega'}(P)\right|
\le\frac{2B}{\deg u}
\le \frac{2B}{2^n}\longrightarrow0.
\]
Thus every open interval containing \(\widehat h_\omega(P)\) contains another such height.
\end{proof}
We conclude this section with a criterion that converts a uniform
orbitwise height gap into a positive Hausdorff-dimension bound.
\begin{lemma}\label{lem:orbitwise-gap}
Let \(\phi,\psi\in K(t)\) have degrees \(d_\phi,d_\psi\ge2\),
where \(K\) is a number field, and let
\(P\in\mathbb P^1(\overline{\mathbb Q})\). Write
\(\mathcal O(P)=\{u(P):u\in\langle\phi,\psi\rangle\}\), including
the empty prefix, and suppose that
\[
\Delta_*(P):=\inf_{y\in\mathcal O(P)}
|\widehat h_\phi(y)-\widehat h_\psi(y)|>0.
\]
Choose \(C_\phi,C_\psi\ge0\) such that
\[
|h(\theta(Q))-d_\theta h(Q)|\le C_\theta
\qquad
\bigl(\theta\in\{\phi,\psi\},\ Q\in\mathbb P^1(\overline{\mathbb Q})\bigr),
\]
and put
\[
B=\max\left\{\frac{C_\phi}{d_\phi-1},
               \frac{C_\psi}{d_\psi-1}\right\}.
\]
Let \(s_0>0\) satisfy
\(d_\phi^{-s_0}+d_\psi^{-s_0}=1\). For any integer \(j\ge1\)
satisfying
\[
2B(d_\phi^{-j}+d_\psi^{-j})<\Delta_*(P),
\]
put
\[
t=\frac{s_0}{j},\qquad
\gamma=\Delta_*(P)-2B(d_\phi^{-j}+d_\psi^{-j})>0.
\]
Then
\[
\dim_H\widehat X_{\phi,\psi}(P)\ge t>0,
\]
and
\[
\mathcal H^t\bigl(\widehat X_{\phi,\psi}(P)\bigr)
\ge\mathcal H^t\bigl(\widehat X_{\phi^j,\psi^j}(P)\bigr)
\ge d_\phi^{-s_0}\gamma^t>0.
\]
\end{lemma}
\begin{proof}
Set \(t=s_0/j>0\) and put
\[
\gamma=\Delta_*(P)
-2B\left(d_\phi^{-j}+d_\psi^{-j}\right)>0.
\]

Consider infinite sequences formed from \(\phi^j,\psi^j\), with the first block fixed to be \(\phi^j\). We expand each block into \(j\) original generators when writing \(\widehat h_\omega(P)\). Since \(X_{\phi^j,\psi^j}(P)\subseteq X_{\phi,\psi}(P)\), taking closures gives
\[
\widehat X_{\phi^j,\psi^j}(P)
\subseteq\widehat X_{\phi,\psi}(P).
\]

Take two distinct sequences \(\omega\) and \(\omega'\) under consideration, and let \(u\) be the longest finite word in the blocks \(\phi^j,\psi^j\) with which both sequences begin. Since their first block is fixed, \(u\) is nonempty, and thus \(u(P)\in\mathcal O(P)\). Their remaining sequences, denoted by \(\alpha,\beta\), begin with \(\phi^j,\psi^j\), respectively.

By Lemma~\ref{lem:height-ping-pong}, applied at \(u(P)\) with both exponents equal to \(j\),
\[
\begin{aligned}
\left|\widehat h_\alpha(u(P))-\widehat h_\beta(u(P))\right|
&\ge
\Delta(u(P))
-2B\left(d_\phi^{-j}+d_\psi^{-j}\right)\\
&\ge\gamma.
\end{aligned}
\]
Removing the common prefix scales heights by \((\deg u)^{-1}\). Therefore
\begin{equation}\label{eq:orbit-separation}
\left|\widehat h_\omega(P)-\widehat h_{\omega'}(P)\right|
\ge\frac{\gamma}{(\deg u)}.
\end{equation}
In particular, distinct sequences under consideration give distinct heights. Also, \(\phi^j\ne\psi^j\), since otherwise \(\widehat h_\phi=\widehat h_{\phi^j}=\widehat h_{\psi^j}=\widehat h_\psi\), contrary to \(\Delta_*(P)>0\).

For each prefix \(u\) beginning with \(\phi^j\), denote $\deg u$ by $d_u$, and consider the number
\[
d_\phi^{s_0}(d_u)^{-t}.
\]
For the initial prefix \(\phi^j\), this number is \(1\). Extending a prefix by either possible next block gives
\begin{equation}\label{eq:prefix-sum}
\begin{aligned}
&d_\phi^{s_0}(d_ud_\phi^j)^{-t}
+d_\phi^{s_0}(d_ud_\psi^j)^{-t}\\
&\qquad=
d_\phi^{s_0}d_u^{-t}
\left(d_\phi^{-s_0}+d_\psi^{-s_0}\right)\\
&\qquad=d_\phi^{s_0}d_u^{-t}.
\end{aligned}
\end{equation}
Consequently, whenever finitely many prefixes cover all the sequences under consideration, the sum of their associated numbers is at least \(1\), the number associated with the initial prefix. Indeed, replace each prefix by all its extensions to a common length. Repeated application of \eqref{eq:prefix-sum} shows that the sum over these extensions equals the number associated with the original prefix. Every word of that length beginning with the fixed first block is counted at least once.

Now take a countable cover
\[
\widehat X_{\phi^j,\psi^j}(P)\subseteq\bigcup_i U_i.
\]
If \(\sum_i(\operatorname{diam}U_i)^t=\infty\), there is nothing to prove. Otherwise, for any \(\varepsilon>0\), enlarge the nonempty \(U_i\) to open intervals \(V_i\) such that
\begin{equation}\label{eq:open-intervals}
\sum_i(\operatorname{diam}V_i)^t
\le
\sum_i(\operatorname{diam}U_i)^t+\varepsilon.
\end{equation}

Indeed, continuity of \(r\mapsto r^t\) allows each enlargement to increase
the corresponding summand by at most \(\varepsilon/2^i\); summing these increases
gives at most \(\varepsilon\).

Discard the intervals \(V_i\) that do not meet the sequence heights under consideration. By Lemma~\ref{lem:sequence-no-isolated-points}, applied to \(\phi^j,\psi^j\), every remaining \(V_i\) contains at least two such heights. Let \(w_i\) be the longest prefix shared by all sequences under consideration whose heights lie in \(V_i\). Two of these sequences diverge immediately after \(w_i\), so \eqref{eq:orbit-separation} gives
\[
\operatorname{diam}V_i\ge\frac{\gamma}{\deg{w_i}}.
\]
It follows that
\begin{equation}\label{eq:interval-bound}
d_\phi^{s_0}(\deg{w_i})^{-t}
\le
d_\phi^{s_0}\gamma^{-t}(\operatorname{diam}V_i)^t.
\end{equation}

Every sequence under consideration begins with one of the prefixes \(w_i\). By compactness of the family of sequences with first block fixed, finitely many of these prefixes suffice. Using the prefix-cover consequence of \eqref{eq:prefix-sum}, followed by \eqref{eq:interval-bound} and \eqref{eq:open-intervals}, we obtain
\[
\begin{aligned}
1
&\le \sum_i d_\phi^{s_0}(\deg{w_i})^{-t}\\
&\le d_\phi^{s_0}\gamma^{-t}
\sum_i(\operatorname{diam}V_i)^t\\
&\le d_\phi^{s_0}\gamma^{-t}
\left(\sum_i(\operatorname{diam}U_i)^t+\varepsilon\right).
\end{aligned}
\]
Letting \(\varepsilon\to0\), taking the infimum over covers of diameter at most \(\delta\), and then letting \(\delta\to0^+\), gives
\[
\mathcal H^t\!\left(\widehat X_{\phi^j,\psi^j}(P)\right)
\ge d_\phi^{-s_0}\gamma^t>0.
\]
Thus
\[
\dim_H\widehat X_{\phi,\psi}(P)
\ge
\dim_H\widehat X_{\phi^j,\psi^j}(P)
\ge t=\frac{s_0}{j}.
\]
\end{proof}
\section{General criteria}\label{sec:general-criteria}
For $\theta\in\{\phi,\psi\}$ and $Q\in\mathbb P^1(K_v)$, set
\[
\kappa_{\theta,v}(Q)
\coloneqq\widehat h_{\theta,v}(\widetilde Q)-\log\|\widetilde Q\|_v,
\]
where $\widetilde Q$ is any nonzero homogeneous vector representing $Q$. Equation~\ref{eq:scaling-identity} shows that this local correction is independent of the representative, although it depends on the fixed lift.

Fix $P\in\mathbb P^1(K)$. Suppose that there are a finite set \(S\subseteq M_K\) and constants \(\eta>\varepsilon_{\mathrm{out}}\ge0\) such that for every \(Q\in\mathcal O(P)\),
\[
\left|
\frac1{[K:\mathbb Q]}
\sum_{v\in S}n_v
\bigl(\kappa_{\phi,v}(Q)-\kappa_{\psi,v}(Q)\bigr)
\right|\ge\eta
\]
and
\[
\left|
\frac1{[K:\mathbb Q]}
\sum_{v\notin S}n_v
\bigl(\kappa_{\phi,v}(Q)-\kappa_{\psi,v}(Q)\bigr)
\right|\le \varepsilon_{\mathrm{out}}.
\]
Then we have $\Delta_*(P)\ge\eta-\varepsilon_{\mathrm{out}}>0$ by the triangle inequality, and this observation directly implies
\[
\dim_H\widehat X_{\phi,\psi}(P)>0
\]
by Lemma~\ref{lem:orbitwise-gap}.

\begin{lemma}\label{lem:common-critical-neighborhood}
Let \(\phi,\psi\in K(t)\) have degrees at least \(2\). Suppose that \(\alpha\in\mathbb P^1(K)\) is periodic and critical for both maps. Choose \(r,s\ge1\) such that
\[
\phi^r(\alpha)=\psi^s(\alpha)=\alpha.
\]
Then, for every place \(v\in M_K\), there is a closed coordinate disk \(D_v\subseteq\mathbb P^1(K_v)\) centered at \(\alpha\) such that
\[
\phi^r(D_v)\subseteq D_v,
\qquad
\psi^s(D_v)\subseteq D_v.
\]
Moreover, every sequence of applications of \(\phi^r\) and \(\psi^s\) to a point of \(D_v\) converges to \(\alpha\).
\end{lemma}
\begin{proof}
Choose a local coordinate \(\chi\) centered at \(\alpha\): take \(\chi=t-\alpha\) when \(\alpha\) is finite, and \(\chi=1/t\) otherwise.

Both iterates fix \(\alpha\) and have local degree at least \(2\) there. Consequently, for some \(M_v>0\) and all sufficiently small \(|\chi(Q)|_v\),
\[
\max\left\{
|\chi(\phi^r(Q))|_v,\,
|\chi(\psi^s(Q))|_v
\right\}
\le M_v|\chi(Q)|_v^2.
\]
Choose \(\rho_v>0\) sufficiently small that this estimate holds on
\[
D_v=\{Q:|\chi(Q)|_v\le\rho_v\}
\]
and \(M_v\rho_v\le\tfrac12\). Then, for \(Q\in D_v\),
\[
|\chi(\phi^r(Q))|_v,\ |\chi(\psi^s(Q))|_v
\le\tfrac12|\chi(Q)|_v.
\]
Thus both iterates preserve \(D_v\), and each successive application multiplies the coordinate size by at most \(1/2\). This proves convergence.
\end{proof}

\begin{lemma}\label{lem:local-height-at-periodic-critical-point}
Let \(\alpha,r,s,D_v\) be as in Lemma~\ref{lem:common-critical-neighborhood}, with a nonarchimedean place \(v\). Choose \(\widetilde\alpha\in K^2\setminus\{0\}\) representing \(\alpha\), and write
\[
\widetilde\phi^{r}(\widetilde\alpha)=\lambda_\phi\widetilde\alpha,
\qquad
\widetilde\psi^{s}(\widetilde\alpha)=\lambda_\psi\widetilde\alpha.
\]
After shrinking \(D_v\), for every \(P\in D_v\) and every nonzero homogeneous representative \(\widetilde P\in K_v^2\),
\[
\widehat h_{\phi,v}(\widetilde P)
=
\log\|\widetilde P\|_v-\log\|\widetilde\alpha\|_v
+\frac{\log|\lambda_\phi|_v}{d_\phi^r-1},
\]
and
\[
\widehat h_{\psi,v}(\widetilde P)
=
\log\|\widetilde P\|_v-\log\|\widetilde\alpha\|_v
+\frac{\log|\lambda_\psi|_v}{d_\psi^s-1}.
\]
In particular,
\[
\widehat h_{\phi,v}(\widetilde P)
-\widehat h_{\psi,v}(\widetilde P)
=
\frac{\log|\lambda_\phi|_v}{d_\phi^r-1}
-\frac{\log|\lambda_\psi|_v}{d_\psi^s-1}.
\]
\end{lemma}
\begin{proof}
Put \(G=\widetilde\phi^{r}\) and \(D=d_\phi^r\). The function
\[
b(P)=\log\|G(\widetilde P)\|_v-D\log\|\widetilde P\|_v
\]
is independent of the representative and locally constant near \(\alpha\). Shrinking \(D_v\) while retaining invariance, we have
\[
b(P)=b(\alpha)
=\log|\lambda_\phi|_v-(D-1)\log\|\widetilde\alpha\|_v.
\]
Since every projective iterate remains in \(D_v\),
\[
\frac{\log\|G^{ n}(\widetilde P)\|_v}{D^n}
-\log\|\widetilde P\|_v
=
b(\alpha)\sum_{j=1}^{n}D^{-j}.
\]
Letting \(n\to\infty\), and using \(G^{ n}=\widetilde\phi^{ rn}\), gives
\[
\widehat h_{\phi,v}(\widetilde P)-\log\|\widetilde P\|_v
=\frac{b(\alpha)}{D-1}.
\]
This proves the first formula. The argument for \(\psi\) is identical, after shrinking \(D_v\) further if necessary. Subtracting the formulas proves the final assertion.
\end{proof}

\begin{lemma}\label{lem:local-height-at-periodic-critical-point-archimedean}
Let \(\alpha,r,s,D_v\) be as in Lemma~\ref{lem:common-critical-neighborhood}, with \(v\) archimedean. Choose \(\widetilde\alpha\in K^2\setminus\{0\}\) representing \(\alpha\), and define \(\lambda_\phi,\lambda_\psi\) as in Lemma~\ref{lem:local-height-at-periodic-critical-point}. Choose a local coordinate \(\chi\) centered at \(\alpha\). After shrinking the common invariant disk \(D_v\), there is \(A_v>0\) such that, for every \(P\in D_v\) and every nonzero homogeneous representative \(\widetilde P\in K_v^2\),
\begin{equation}\label{eq:bound-1}
\left|
\widehat h_{\phi,v}(\widetilde P)
-\log\|\widetilde P\|_v+\log\|\widetilde\alpha\|_v
-\frac{\log|\lambda_\phi|_v}{d_\phi^r-1}
\right|
\le
\frac{A_v|\chi(P)|_v}{d_\phi^r-1},
\end{equation}
and
\begin{equation}\label{eq:bound-2}
\left|
\widehat h_{\psi,v}(\widetilde P)
-\log\|\widetilde P\|_v+\log\|\widetilde\alpha\|_v
-\frac{\log|\lambda_\psi|_v}{d_\psi^s-1}
\right|
\le
\frac{A_v|\chi(P)|_v}{d_\psi^s-1}.
\end{equation}
\end{lemma}
\begin{proof}
Put \(G=\widetilde\phi^{r}\), \(D=d_\phi^r\), and define
\[
b(P)=\log\|G(\widetilde P)\|_v-D\log\|\widetilde P\|_v.
\]
In a local coordinate chart, choose representatives depending analytically on \(\chi\). Their norms and image norms are bounded away from zero near \(\alpha\). Consequently, their logarithmic norms are locally Lipschitz, so
\[
|b(P)-b(\alpha)|\le A_v|\chi(P)|_v
\]
on a sufficiently small disk. Choose \(A_v\) to work for both return maps.

Write \(P_k=\phi^{rk}(P)\). The invariant-disk estimate gives
\[
|\chi(P_k)|_v\le |\chi(P)|_v.
\]
Similar to nonarchimedean cases, we have
\[
\frac{\log\|G^{n}(\widetilde P)\|_v}{D^n}
-\log\|\widetilde P\|_v
=
\sum_{k=0}^{n-1}\frac{b(P_k)}{D^{k+1}}.
\]
Subtracting \(b(\alpha)\sum_{k=0}^{n-1}D^{-k-1}\), the remaining sum has absolute value at most
\[
\sum_{k=0}^{n-1}
\frac{|b(P_k)-b(\alpha)|}{D^{k+1}}
\le
\frac{A_v|\chi(P)|_v}{D-1}.
\]
Letting \(n\to\infty\), we obtain
\[
\left|
\widehat h_{\phi,v}(\widetilde P)
-\log\|\widetilde P\|_v-\frac{b(\alpha)}{D-1}
\right|
\le\frac{A_v|\chi(P)|_v}{D-1}.
\]
Now
\[
b(\alpha)
=\log|\lambda_\phi|_v-(D-1)\log\|\widetilde\alpha\|_v,
\]
which proves the first estimate. The argument for \(\psi\) is identical.
\end{proof}

\begin{proof}[Proof of Theorem~\ref{thm:main-1}]
The expression defining \(\Lambda\) is independent of the representatives: rescaling \(\widetilde\alpha_v\) adds the same logarithm to its two normalized eigenvalue terms.

Choose common invariant disks using Lemma~\ref{lem:common-critical-neighborhood}. At each finite place in \(S\), Lemma~\ref{lem:local-height-at-periodic-critical-point} gives
\[
\widehat h_{\phi,v}(\widetilde Q)
-\widehat h_{\psi,v}(\widetilde Q)
=
\frac{\log|\lambda_{\phi,v}|_v}{d_\phi^r-1}
-
\frac{\log|\lambda_{\psi,v}|_v}{d_\psi^s-1}
\]
for $Q\in D_v$. At each archimedean place, Lemma~\ref{lem:local-height-at-periodic-critical-point-archimedean} gives the same identity with an error tending uniformly to zero as \(D_v\) shrinks. Choose these disks so that the bounds in \eqref{eq:bound-1} and \eqref{eq:bound-2}, multiplied by \(n_v/[K:\mathbb Q]\) and summed over all archimedean places, have total at most \(|\Lambda|/2\).

Outside \(S\), good reduction of the fixed lifts gives
\[
\widehat h_{\phi,v}(\widetilde Q)
=\log\|\widetilde Q\|_v
=\widehat h_{\psi,v}(\widetilde Q).
\]
The global decomposition~\eqref{eq:global-local-decomposition} therefore yields, for $Q\in\mathcal U$,
\[
\left|
\widehat h_\phi(Q)-\widehat h_\psi(Q)-\Lambda
\right|
\le\frac{|\Lambda|}{2}.
\]
Consequently,
\[
|\widehat h_\phi(Q)-\widehat h_\psi(Q)|
\ge\frac{|\Lambda|}{2}>0.
\]

Both $\phi^r$ and $\psi^s$ preserve \(\mathcal U\), and
\[
\widehat h_{\phi^r}=\widehat h_\phi,
\qquad
\widehat h_{\psi^s}=\widehat h_\psi.
\]
Lemma~\ref{lem:orbitwise-gap}, applied to \(\phi^r,\psi^s\), gives
\[
\dim_H\widehat X_{\phi^r,\psi^s}(P)>0.
\]
Since this completed spectrum is contained in \(\widehat X_{\phi,\psi}(P)\), the conclusion follows.

Finally, the disks contain open neighborhoods at finitely many places, so weak approximation gives \(\mathcal U\ne\varnothing\). See Milne~\cite[Theorem~7.20]{MilneANT2020}.
\end{proof}

\section{Polynomial and rational cases}\label{sec:applications}
\subsection{Polynomial cases}
Throughout this subsection, let $\phi,\psi\in K[t]$ be polynomials
of degree at least $2$, where $K$ is a number field. Write
$a_\phi$ and $a_\psi$ for their leading coefficients, and use
the affine height convention from the preliminaries. For
$v\in M_K$ and $\theta\in\{\phi,\psi\}$, put
\[
c_{\theta,v}
=\frac{\log|a_\theta|_v}{d_\theta-1}.
\]

We now prove the polynomial dimension statements. To prove
that a height spectrum has Hausdorff dimension zero, we cover
height cylinders $\widehat X^u(z)$ by sufficiently small
intervals. To obtain a positive lower bound, we apply
Lemma~\ref{lem:orbitwise-gap}. Lemma~\ref{lem:local-height} specializes the general local estimates of Section~\ref{sec:general-criteria} at infinity. Lemma~\ref{lem:local-sequence-height} then passes to mixed sequences by telescoping, under the common-correction assumption at the selected place.

\begin{lemma}\label{lem:recurrence}
Let $(L_i)_{i\ge0}$ satisfy
\[
L_i=d_iL_{i-1}+b_i,
\]
where $d_i,b_i\in\mathbb R$ and $d_i\ne0$. Put
\[
D_0=1,\qquad D_i=\prod_{j=1}^{i}d_j.
\]
Then, for $0\le m\le n$,
\[
\frac{L_n}{D_n}
=
\frac{L_m}{D_m}
+\sum_{i=m+1}^{n}\frac{b_i}{D_i}.
\]
\end{lemma}

\begin{proof}
Divide the recurrence by $D_i$ and telescope.
\end{proof}

\begin{lemma}\label{lem:local-height}
Fix $v\in M_K$. There exists $R_v\ge1$ such that
\[
|\theta(y)|_v\ge2|y|_v
\qquad
\bigl(\theta\in\{\phi,\psi\},\ |y|_v>R_v\bigr).
\]
For these $y$, we have
\[
\log|\theta(y)|_v+c_{\theta,v}
=
d_\theta\bigl(\log|y|_v+c_{\theta,v}\bigr)
+
\begin{cases}
0,&v\text{ nonarchimedean},\\
O(|y|_v^{-1}),&v\text{ archimedean}.
\end{cases}
\]
The implied constant depends only on $\phi,\psi$ and $v$.
\end{lemma}

\begin{proof}
Infinity is a critical fixed point of both polynomials. Apply Lemma~\ref{lem:common-critical-neighborhood} with $\alpha=\infty$, $r=s=1$, and local coordinate $\chi=1/t$. Its coordinate contraction gives
\[
|\theta(y)|_v\ge2|y|_v
\]
for sufficiently large $|y|_v$, simultaneously for both generators.
Use the polynomial lifts fixed in Section~\ref{subsec:local-sequence-height}. At $\widetilde\alpha=(1,0)$ their eigenvalues are $a_\theta$. Lemmas~\ref{lem:local-height-at-periodic-critical-point} and~\ref{lem:local-height-at-periodic-critical-point-archimedean} therefore give
\[
\widehat h_{\theta,v}(y)
=\log|y|_v+c_{\theta,v}+e_\theta(y),
\]
where $e_\theta=0$ at a finite place and $e_\theta(y)=O(|y|_v^{-1})$ at an archimedean place. Both estimates hold throughout a common exterior region.

The defining limit gives
\[
\widehat h_{\theta,v}(\widetilde\theta(\widetilde P))
=d_\theta\widehat h_{\theta,v}(\widetilde P).
\]
Since $\widetilde\theta(y,1)=(\theta(y),1)$, substitution yields
\[
\log|\theta(y)|_v+c_{\theta,v}
-d_\theta(\log|y|_v+c_{\theta,v})
=d_\theta e_\theta(y)-e_\theta(\theta(y)).
\]
The right-hand side is zero at a finite place. At an archimedean place it is $O(|y|_v^{-1})$, because $|\theta(y)|_v\ge2|y|_v$. The degrees and local constants depend only on the fixed maps and $v$, as required.
\end{proof}

\begin{lemma}\label{lem:local-sequence-height}
Fix $v\in M_K$, and suppose that
\[
c_{\phi,v}=c_{\psi,v}=:c_v.
\]
Let $z\in K$ and
$\omega=(\theta_1,\theta_2,\ldots)\in\{\phi,\psi\}^{\mathbb N}$.
Write
\[
\omega_0=\operatorname{id},\qquad
\omega_i=\theta_i\circ\cdots\circ\theta_1,\qquad
D_i=\deg\omega_i.
\]
Let $R_v$ be as in Lemma~\ref{lem:local-height}. If
$|\omega_m(z)|_v>R_v$ for some $m\ge0$, then
\[
\widehat h_{\omega,v}(z)
=
\frac{\log|\omega_m(z)|_v+c_v}{D_m}
\]
at a nonarchimedean place. At an archimedean place,
\[
\left|
\widehat h_{\omega,v}(z)
-\frac{\log|\omega_m(z)|_v+c_v}{D_m}
\right|
\le
\frac{A_v}{D_m|\omega_m(z)|_v},
\]
where $A_v$ depends only on $\phi,\psi$ and $v$.

If the orbit never leaves the disk $|y|_v\le R_v$,
its local sequence height is zero.
\end{lemma}

\begin{proof}
Once the orbit leaves the disk,
Lemma~\ref{lem:local-height} shows that its absolute value
increases by a factor of at least $2$ at each step.

At a nonarchimedean place, Lemma~\ref{lem:recurrence} gives
\[
\frac{\log|\omega_n(z)|_v+c_v}{D_n}
=
\frac{\log|\omega_m(z)|_v+c_v}{D_m}
\qquad(n\ge m).
\]
Letting $n\to\infty$ proves the first assertion.

At an archimedean place, the same telescoping argument and
Lemma~\ref{lem:local-height} give
\[
\begin{aligned}
\left|
\frac{\log|\omega_n(z)|_v+c_v}{D_n}
-\frac{\log|\omega_m(z)|_v+c_v}{D_m}
\right|
&\le
A_v\sum_{i=m+1}^{n}
\frac{1}{D_i|\omega_{i-1}(z)|_v}\\
&\le
\frac{A_v}{|\omega_m(z)|_v}
\sum_{i=m+1}^{n}\frac1{D_i}\\
&\le
\frac{A_v}{D_m|\omega_m(z)|_v},
\end{aligned}
\]
since every degree is at least $2$.
The escaped orbit tends to infinity, and $c_v/D_n\to0$.
Letting $n\to\infty$ therefore gives the required estimate.

Finally, if the orbit stays in the disk, its logarithmic
positive part is bounded while $D_n\to\infty$. Its local
sequence height is therefore zero.
\end{proof}

The local estimates show how a common escaping prefix controls differences between sequence heights. When $c_{\phi,v}=c_{\psi,v}$, the nonarchimedean contributions agree after such a prefix, while the archimedean errors decrease as its image grows. When $c_{\phi,v}\ne c_{\psi,v}$, the one-map estimates give the local gap used in the positive-dimension argument below.

\begin{proof}[Proof of Theorem~\ref{thm:main-2}\textup{(i)}]
Fix $z\in\overline{\mathbb Q}$. Enlarging $K$ if necessary,
assume that $z\in K$. Condition~\eqref{eq:zero-1} remains
valid at every place of the extension. Write $c_v$ for the
common value.

Choose the radii $R_v$ as in
Lemma~\ref{lem:local-height}. At each archimedean place,
increase $R_v$ so that $R_v\ge2$ and
\[
|\theta(y)|_v\ge |y|_v^{3/2}
\qquad
\bigl(\theta\in\{\phi,\psi\},\ |y|_v>R_v\bigr).
\]
This is possible because both polynomials have degree
at least two.

Only finitely many places contribute, uniformly in the sequence. Outside the archimedean places and the finite places where $z$ or a coefficient is not integral, both maps preserve the unit disk containing $z$. Group the sequences according to the contributing places where they escape and a common finite prefix $u$ after all these escapes have occurred. There are countably many such groups.

Within each group, Lemma~\ref{lem:local-sequence-height}
shows that the local heights agree at every finite place:
they are determined by the common prefix after escape,
and are zero otherwise. Thus the finite-place contributions
cancel when comparing two heights in the same group.
We need only estimate the archimedean differences.

Fix one group. If no archimedean place escapes, all global heights in this group agree. Otherwise, nonescaping archimedean places contribute zero. After $n$ further maps,
the growth estimate gives
\[
|\omega_{|u|+n}(z)|_v
\ge 2^{(3/2)^n}
\]
at every escaping archimedean place. By
Lemma~\ref{lem:local-sequence-height}, sequences in this
group sharing these $n$ further maps have global heights
in an interval of diameter at most
\[
A\,2^{-(3/2)^n},
\]
where $A>0$ is independent of $n$ and the continuation.

There are at most $2^n$ such extensions. For every $s>0$,
their covering cost is at most
\[
2^n\left(A\,2^{-(3/2)^n}\right)^s
\longrightarrow0.
\]
Thus the heights in each group form a set of Hausdorff dimension zero. Since there are countably many groups, the set of all sequence heights also has Hausdorff dimension zero.

Finally, the sequence-height map is continuous by the
cylinder estimate, and periodic sequences are dense in
the compact sequence space. Its image is therefore
$\widehat X_{\phi,\psi}(z)$, proving part~\textup{(i)} of Theorem~\ref{thm:main-2}.
\end{proof}

\begin{proof}[Proof of Theorem~\ref{thm:main-2}\textup{(ii)}]
For $\theta\in\{\phi,\psi\}$, let
$\widehat h_{\theta,v}$ denote the local height associated
to the constant sequence $(\theta,\theta,\ldots)$.

Apply Lemmas~\ref{lem:local-height-at-periodic-critical-point} and~\ref{lem:local-height-at-periodic-critical-point-archimedean} at $\alpha=\infty$, with $r=s=1$, $\widetilde\alpha=(1,0)$ and the normalized polynomial lifts. Since $\lambda_\theta=a_\theta$ and $\chi=1/t$, we obtain
\begin{equation}\label{eq:positive-2}
\widehat h_{\theta,v}(y)
=
\log|y|_v+c_{\theta,v}
+
\begin{cases}
0,&v\text{ nonarchimedean},\\
O(|y|_v^{-1}),&v\text{ archimedean},
\end{cases}
\end{equation}
for sufficiently large $|y|_v$. The implied constant
depends only on $\phi,\psi$ and $v$.

Since $\delta_v\ne0$, we may choose $R\ge R_v$ so that
\begin{equation}\label{eq:positive-5}
\left|
\widehat h_{\phi,v}(y)-\widehat h_{\psi,v}(y)
\right|
\ge\frac{|\delta_v|}{2}
\qquad\text{with }|y|_v>R.
\end{equation}

Fix a point $z$ satisfying the hypotheses. The growth
estimate in Lemma~\ref{lem:local-height} shows that
$|y|_v>R$ for every $y\in\mathcal O(z)$.

At every other place $v'$, the iterates of any such $y$
under either individual generator remain in
$\mathcal O(z)$. The boundedness hypothesis therefore gives
\begin{equation}\label{eq:positive-6}
\widehat h_{\phi,v'}(y)
=
\widehat h_{\psi,v'}(y)
=0.
\end{equation}

Only finitely many places contribute to the local
decomposition \eqref{eq:global-local-decomposition}, uniformly over the orbit and its further
iterates: outside the archimedean places and the finite
places where $z$ or a coefficient is not integral, both
maps preserve the unit disk. Consequently,
\begin{equation}\label{eq:positive-7}
\widehat h_\theta(y)
=
\frac{n_v}{[K:\mathbb Q]}\widehat h_{\theta,v}(y)
\qquad
\text{with }\theta\in\{\phi,\psi\},\ y\in\mathcal O(z).
\end{equation}

Combining~\eqref{eq:positive-5}--\eqref{eq:positive-7},
we obtain
\begin{equation}\label{eq:positive-8}
\Delta_*(z)
=
\inf_{y\in\mathcal O(z)}
\left|\widehat h_\phi(y)-\widehat h_\psi(y)\right|
\ge
\frac{n_v|\delta_v|}{2[K:\mathbb Q]}
>0.
\end{equation}
Lemma~\ref{lem:orbitwise-gap} now gives the conclusion.

More explicitly, let $B$ be the constant from
Lemma~\ref{lem:sequence-bound}, and let $s_0>0$ satisfy
\[
d_\phi^{-s_0}+d_\psi^{-s_0}=1.
\]
Choose an integer $j\ge1$ such that
\[
2B\left(d_\phi^{-j}+d_\psi^{-j}\right)
<
\frac{n_v|\delta_v|}{2[K:\mathbb Q]}.
\]
Then
\[
\dim_H\widehat X_{\phi,\psi}(z)\ge\frac{s_0}{j}>0.
\]
Both $R$ and this positive lower bound can be chosen
independently of the admissible starting point $z$.
\end{proof}
\begin{corollary}\label{cor:intermediate-dimension}
Let \(\phi,\psi\in\mathbb Z[t]\) have positive leading coefficients
\(a_\phi,a_\psi\) and degrees \(d_\phi,d_\psi\ge2\). Suppose that
\[
\frac1{d_\phi}+\frac1{d_\psi}<1,
\qquad
\frac{\log a_\phi}{d_\phi-1}\ne
\frac{\log a_\psi}{d_\psi-1}.
\]
Then there exists \(R>0\) such that, for every integer \(x\ge R\),
\[
0<\dim_H\widehat X_{\phi,\psi}(x)\le s_0<1,
\qquad d_\phi^{-s_0}+d_\psi^{-s_0}=1.
\]
\end{corollary}

\begin{proof}
At every finite place, both polynomials preserve the unit disk.
The entire orbit of an integer is therefore bounded there.
Theorem~\ref{thm:main-2}\textup{(ii)}, applied at the archimedean
place of \(\mathbb Q\), gives positive dimension for every sufficiently
large positive integer \(x\). The upper bound follows from
Proposition~\ref{prop:hausdorff-upper-bound}. Finally, the function
\(s\mapsto d_\phi^{-s}+d_\psi^{-s}\) is strictly decreasing,
has value \(2\) at \(s=0\), and has value less than \(1\) at \(s=1\).
Consequently, \(0<s_0<1\).
\end{proof}

\subsection{Rational cases}\label{subsec:rational-cases}

We apply Theorem~\ref{thm:main-1} to a rational family that is not conjugate to a polynomial family by a m\"{o}bius transformation.

\begin{proposition}\label{prop:rational-cubic-spectrum}
Let \(a,b\ge3\) be distinct integers, and put
\[
\theta_c(t)=\frac{ct^2}{1+(c-3)t^2+2t^3}
\]
with $c\in\{a,b\}$. For every prime \(p\mid2ab\), choose an integer
\[
m_p>\max\{v_p(a),v_p(b),0\},
\qquad
L=\prod_{p\mid2ab}p^{m_p}.
\]
Then, for every positive integer \(T\) and
\(x=1/(1+LT)\),
\[
\dim_H\widehat X_{\theta_a,\theta_b}(x)
=\rho:=\frac{\log2}{\log3},
\qquad
0<\mathcal H^\rho\bigl(\widehat X_{\theta_a,\theta_b}(x)\bigr)<\infty.
\]
\end{proposition}

\begin{proof}
Use the fixed global lifts
\[
\widetilde\theta_c(X,Y)=\bigl(cX^2Y,\ Y^3+(c-3)X^2Y+2X^3\bigr).
\]
Their resultants have absolute value \(2c^3\).
Both \(0\) and \(1\) are critical fixed points, and
\[
\widetilde\theta_c(0,1)=(0,1),\qquad \widetilde\theta_c(1,1)=c(1,1).
\]
Choosing center \(0\) at the real place and center \(1\) at
every prime dividing \(2ab\), the constant in
Theorem~\ref{thm:main-1} is
\[
\Lambda
=\frac12\sum_{p\mid2ab}\bigl(\log|a|_p-\log|b|_p\bigr)
=-\frac12\log(a/b)\ne0.
\]
We give uniform estimates on an explicit invariant set to obtain
the stated points and the exact dimension.

Put \(N=\max(a,b)\), \(r=1/(2N)\), and
\[
\varepsilon_{\mathrm{cub}}=\log\bigl(1+(N-3)r^2+2r^3\bigr)<\frac1{4N}.
\]
Consider rational points satisfying
\[
0<y\le r,\qquad
y\in1+p^{m_p}\mathbb Z_p\quad(p\mid2ab).
\tag{*}\label{eq:cubic-invariant-locus}
\]
The integer \(L\) is divisible by \(2\operatorname{lcm}(a,b)\),
so \(L\ge2N\). Consequently every stated \(x\) satisfies
\eqref{eq:cubic-invariant-locus}. In particular, no further lower
bound on the positive integer \(T\) is required.

Write \(B_c(t)=1+(c-3)t^2+2t^3\). For real \(0<y\le r\),
\[
0<\theta_c(y)\le cy^2\le y/2.
\]
At a selected finite place, the identities
\[
\begin{split}
B_c(1+t)&=c+2ct+(c+3)t^2+2t^3,\\
\theta_c(y)-1&=-\frac{(y-1)^2(2y+1)}{B_c(y)}
\end{split}
\]
show that
\[
v_p(B_c(y))=v_p(c),\qquad
v_p(\theta_c(y)-1)\ge2m_p-v_p(c)>m_p.
\]
Thus both maps preserve the set in
\eqref{eq:cubic-invariant-locus}.

For \(y\) in this set, its representative \((y,1)\) has
norm one at the real place and the selected finite places.
At the real place, \(B_c(y)\ge1>cy^2\), so
\[
\log\|\widetilde\theta_c(y,1)\|_\infty-3\log\|(y,1)\|_\infty
=\log B_c(y)\in[0,\varepsilon_{\mathrm{cub}}].
\]
At a selected prime the corresponding defect is
\(-v_p(c)\log p\). At every other finite place the fixed lift
has integral coefficients and unit resultant, so its norm
defect is zero for every nonzero vector. Summing the defects
in the global height formula gives
\begin{equation}\label{eq:cubic-height-recurrence}
h(\theta_c(y))=3h(y)-\log c+e_c(y),
\qquad 0\le e_c(y)\le \varepsilon_{\mathrm{cub}}.
\end{equation}
At the remaining places, the unchanged norm contributions cancel in the height difference.

For an infinite sequence with parameter letters
\(c_i\in\{a,b\}\), telescoping
\eqref{eq:cubic-height-recurrence} gives
\[
\widehat h_\omega(y)
=h(y)-\sum_{i\ge1}\frac{\log c_i}{3^i}+R_\omega(y),
\qquad 0\le R_\omega(y)\le \varepsilon_{\mathrm{cub}}/2.
\]
Assume \(a<b\), interchanging the labels if necessary.
The model values beginning with \(a\) lie in
\[
\left[-\frac{\log a}{3}-\frac{\log b}{6},
      -\frac{\log a}{2}\right],
\]
whereas those beginning with \(b\) lie in
\[
\left[-\frac{\log b}{2},
      -\frac{\log b}{3}-\frac{\log a}{6}\right].
\]
Their distance is \(\log(b/a)/6\). The residuals lie in
the same interval \([0,\varepsilon_{\mathrm{cub}}/2]\), so the actual first-letter
spectra are separated by at least
\[
\gamma=\frac{|\log(a/b)|}{6}-\frac{\varepsilon_{\mathrm{cub}}}{2}
>\frac1{24N}>0.
\]
Here \(|\log(a/b)|\ge1/N\). Any two sequence heights at
the same \(y\) differ by at most
\[
C=\frac{|\log(a/b)|}{2}+\frac{\varepsilon_{\mathrm{cub}}}{2}.
\]
Since the entire orbit stays in the invariant set, prefix
scaling shows that sequences first differing after a prefix
\(u\) satisfy
\begin{equation}\label{eq:cubic-coding-separation}
\gamma\,3^{-|u|}
\le
|\widehat h_\omega(x)-\widehat h_{\omega'}(x)|
\le C\,3^{-|u|}.
\end{equation}

The sequence-height image equals the completed finite-composition spectrum by continuity and density of periodic sequences.

Apply the covering argument in the proof of Lemma~\ref{lem:orbitwise-gap} to the preceding separation estimate, using the numbers
\[
2^{-|u|}=3^{-\rho|u|}.
\]
This gives
\[
\mathcal H^\rho\!\left(\widehat X_{\theta_a,\theta_b}(x)\right)>0.
\]
On the other hand, the \(2^n\) cylinders of length \(n\) have diameter at most \(C3^{-n}\). Since
\[
2^n(C3^{-n})^\rho=C^\rho,
\]
we obtain
\[
0<\mathcal H^\rho\!\left(\widehat X_{\theta_a,\theta_b}(x)\right)\le C^\rho<\infty.
\]
Consequently, the Hausdorff dimension is \(\rho\).

\end{proof}

The maps in Proposition~\ref{prop:rational-cubic-spectrum} are not M\"obius-conjugate to polynomials. Indeed, the derivative numerator of \(\theta_c\) is
\(2ct(1-t^3)\). The four zeros of this numerator are simple and are not poles;
the poles are simple, and infinity has local degree one. Thus every local degree is at most two. A cubic polynomial
has local degree three at infinity, and local degree is
preserved by M\"obius conjugacy.

For example, take \(a=3\), \(b=4\), \(m_2=3\), and
\(m_3=2\). Then \(L=72\), so \(T=1\) gives
\[
\phi(t)=\frac{3t^2}{1+2t^3},\qquad
\psi(t)=\frac{4t^2}{1+t^2+2t^3},\qquad x=\frac1{73}.
\]
Here \(r=1/8\), \(\varepsilon_{\mathrm{cub}}=\log(261/256)\), and
\(\Lambda=\tfrac12\log(4/3)\). The completed spectrum
has dimension \(\log2/\log3\).

\section{Orbit counting}\label{sec:orbit-counting}

We count distinct orbit points by ordinary height. Let \(K\) be a
number field containing \(\phi,\psi\) and \(P\), and write
\[
N_P(X)=\#\{Q\in\mathcal O(P):h(Q)\le X\}.
\]
The orbit includes \(P\). Different words may give the same map,
the same point, or the same height. None of the arguments below
requires these evaluations to be injective. Northcott's theorem
makes \(N_P(X)\) finite.

Let \(s_0>0\) satisfy
\[
d_\phi^{-s_0}+d_\psi^{-s_0}=1,
\qquad d_{\max}=\max\{d_\phi,d_\psi\}.
\]
Choose \(B>0\) as in Lemma~\ref{lem:sequence-bound}.
For \(T>1\), let \(W_T\) consist of the formal words whose degrees
first reach \(T\): every proper initial prefix has degree less
than \(T\). Words are counted separately even if they induce the
same rational map. For a finite word \(u\), put
\[
D_u=\deg u,\qquad
I_u(P)=\left[
\frac{h(u(P))-B}{D_u},
\frac{h(u(P))+B}{D_u}
\right].
\]
The number \(h(u(P))/D_u\) is the normalized center of this interval.

\begin{lemma}\label{lem:stopping-center-cover}
Choose \(B>0\) as in Lemma~\ref{lem:sequence-bound}. For \(T>1\), the intervals
\[
I_u(P)=
\left[
\frac{h(u(P))-B}{\deg u},
\frac{h(u(P))+B}{\deg u}
\right],
\qquad u\in W_T,
\]
cover \(\widehat X_{\phi,\psi}(P)\).

Among intervals with the same normalized center, retain only a largest one, and let \(M_T\) denote the number of retained intervals. Their union is unchanged, and
\[
M_T\le(d_{\max}T)^{s_0},
\qquad
\operatorname{diam}I_u(P)\le\frac{2B}{T}.
\]
Consequently, \(\widehat X_{\phi,\psi}(P)\) has a cover by \(O(\delta^{-s_0})\) intervals of diameter at most \(\delta\) as \(\delta\to0^+\). If \(\mathcal H^{s_0}(\widehat X_{\phi,\psi}(P))>0\), then
\[
M_T\asymp T^{s_0}.
\]
\end{lemma}
\begin{proof}
Every infinite sequence has a first prefix in \(W_T\). The sequence-height description and Lemma~\ref{lem:cylinder-containment} therefore show that the intervals \(I_u(P)\), \(u\in W_T\), cover \(\widehat X_{\phi,\psi}(P)\).

For every finite word \(u\),
\[
\begin{aligned}
&\bigl(\deg(\phi\circ u)\bigr)^{-s_0}
+\bigl(\deg(\psi\circ u)\bigr)^{-s_0}\\
&\qquad=(\deg u)^{-s_0}
\bigl(d_\phi^{-s_0}+d_\psi^{-s_0}\bigr)
=(\deg u)^{-s_0}.
\end{aligned}
\]
Start with the single term \(1\), corresponding to the empty word of degree \(1\). Whenever a word has degree less than \(T\), replace its term by the two terms obtained by appending \(\phi\) and \(\psi\). The sum remains unchanged. Since both generator degrees are at least \(2\), this process terminates after finitely many steps and gives
\[
\sum_{u\in W_T}(\deg u)^{-s_0}=1.
\]
By the definition of \(W_T\),
\[
T\le\deg u<d_{\max}T.
\]
Hence
\[
1\ge\#W_T\,(d_{\max}T)^{-s_0},
\]
so
\[
\#W_T\le(d_{\max}T)^{s_0}.
\]

Intervals with the same center are nested, so retaining only a largest one preserves their union. Thus
\[
M_T\le\#W_T\le(d_{\max}T)^{s_0}.
\]
Every retained interval has diameter at most \(2B/T\). Taking \(T=2B/\delta\) proves the covering estimate.

Finally, suppose \(\mathcal H^{s_0}(\widehat X_{\phi,\psi}(P))>0\). Choose \(c_0,\delta_0>0\) such that
\[
\mathcal H^{s_0}_{\delta_0}(\widehat X_{\phi,\psi}(P))\ge c_0.
\]
For sufficiently large \(T\), the retained intervals have diameter at most \(\delta_0\), and therefore
\[
c_0
\le
\sum_{\text{retained }I_u(P)}
\bigl(\operatorname{diam}I_u(P)\bigr)^{s_0}
\le
M_T\left(\frac{2B}{T}\right)^{s_0}.
\]
Together with the upper bound, this gives
\[
M_T\asymp T^{s_0}.
\]
\end{proof}

We use the following consequence of Bell--Hindes--Zhong~\cite[Theorem~1.5]{BellHindesZhongApplications}.
\begin{lemma}\label{lem:bhz-orbit-upper-bound}
Let \(K\) be a number field, let \(\phi,\psi\in K(t)\) have degrees \(d_\phi,d_\psi\ge2\), and let \(P\in\mathbb P^1(K)\). Let \(s_0>0\) satisfy
\[
d_\phi^{-s_0}+d_\psi^{-s_0}=1.
\]
Then
\[
N_P(X)=\#\{Q\in\mathcal O(P):h(Q)\le X\}
\ll X^{s_0},
\]
where \(h\) denotes the absolute logarithmic Weil height and the implicit constant may depend on \(\phi,\psi,P\).
\end{lemma}

\begin{proof}
For distinct maps, apply the cited theorem with
\(V=\mathbb P^1\) and \(\mathcal L=\mathcal O(1)\).
If \(\phi=\psi\), its orbit is finite or has \(O(\log X)\)
points of logarithmic height at most \(X\), by the usual canonical
height estimate; see Silverman~\cite[Section~3.4]{Silverman2007}.
This also gives the stated bound.
\end{proof}

We begin with a lower bound that applies to every positive
spectrum dimension, without assuming positive Hausdorff measure
at \(s_0\).
\begin{proposition}\label{prop:orbit-counting-lower}
Let \(\phi,\psi\in\overline{\mathbb Q}(t)\) have degrees at least
two, and let \(P\in\mathbb P^1(\overline{\mathbb Q})\). Suppose that
\[
\delta=\dim_H\widehat X_{\phi,\psi}(P)>0,
\qquad
N_P(X)=\#\{y\in\mathcal O(P):h(y)\le X\}.
\]
For every \(0<s<\delta\), there exists \(c_s>0\) such that
\[
N_P(X)\ge c_sX^s\qquad(X\text{ sufficiently large}).
\]
In particular,
\[
\liminf_{X\to\infty}\frac{\log N_P(X)}{\log X}\ge\delta.
\]
If also \(\mathcal H^\delta(\widehat X_{\phi,\psi}(P))>0\), then
\[
N_P(X)\ge c\frac{X^\delta}{\log X}
\qquad(X\text{ sufficiently large}).
\]
Under this additional measure hypothesis, if
\(d_\phi,d_\psi\) are multiplicatively dependent, the lower bound
improves to \(N_P(X)\ge cX^\delta\) for sufficiently large \(X\).
\end{proposition}

\begin{proof}
Fix a number field containing the maps and \(P\). Its orbit
remains in this field, so Northcott's theorem makes every
\(N_P(X)\) finite. Let \(B\) be the constant in
Lemma~\ref{lem:sequence-bound}. Since \(B=0\) would make
every sequence height equal to \(h(P)\), we have \(B>0\).
Put \(\delta=\dim_H\widehat X_{\phi,\psi}(P)>0\).

By Lemma~\ref{lem:stopping-center-cover}, the intervals \(I_u(P)\),
\(u\in W_T\), cover the completed spectrum and have diameter at
most \(2B/T\). From the finite telescoping calculation in the proof
of Lemma~\ref{lem:sequence-bound},
\[
h(u(P))\le D_u(h(P)+B)\le C_0T,
\qquad C_0=d_{\max}(h(P)+B)>0.
\]

Each interval is determined by \((u(P),\deg u)\).
There are only \(O(\log T)\) possible degrees in the
stopping window. Indeed, every degree is
\(d_\phi^k d_\psi^\ell\), with \(k,\ell\ge0\), and
\[
\log T\le k\log d_\phi+\ell\log d_\psi
<\log T+\log d_{\max}.
\]
There are \(O(\log T)\) choices of \(k\), and for each
one the allowable \(\ell\) lie in an interval of fixed
length. Thus, after removing repeated pairs, at most
\(C_1(\log T)N_P(C_0T)\) intervals are needed for large
\(T\). Repeated evaluations have already been accounted
for in this estimate.

Fix \(0<q<\delta\). Some fixed sufficiently small-scale
\(q\)-dimensional Hausdorff content of the spectrum is
bounded below by a positive constant \(c_q\).
Applying it to these covers for large \(T\) gives
\[
c_q\le C_1(\log T)N_P(C_0T)(2B/T)^q.
\]
Replacing \(T\) by \(X/C_0\), we obtain
\[
N_P(X)\ge c'_q X^q/\log X
\qquad(X\text{ sufficiently large}).
\]
Given \(s<\delta\), choose \(s<q<\delta\). Since
\(X^{q-s}/\log X\to\infty\), the first conclusion
follows. Taking lower logarithmic growth exponents and
letting \(s\to\delta\) gives the liminf assertion.

If the critical-dimensional Hausdorff measure is positive,
the same argument applies at \(q=\delta\), giving the
stated critical bound. Finally, multiplicative dependence
makes \(\log d_\phi\) and \(\log d_\psi\) integer
multiples of one positive real number. The stopping window
then contains only a bounded number of distinct log-degrees,
independently of \(T\). The degree-counting factor
\(\log T\) is replaced by a constant, proving the last
assertion under the same critical-measure hypothesis.
\end{proof}

For the multiplicative Weil height \(H=e^h\), the first
bound is equivalently
\[
\#\{y\in\mathcal O(P):H(y)\le Z\}
\ge c_s(\log Z)^s
\qquad(Z\text{ sufficiently large}).
\]

Combining Proposition~\ref{prop:orbit-counting-lower} with Lemma~\ref{lem:bhz-orbit-upper-bound} gives

\[
\dim_H\widehat X_{\phi,\psi}(P)
\le\liminf_{X\to\infty}\frac{\log N_P(X)}{\log X}
\le\limsup_{X\to\infty}\frac{\log N_P(X)}{\log X}
\le s_0
\]
whenever the spectrum has positive dimension. In particular, if its dimension is $s_0$, the logarithmic orbit-counting exponent exists and equals $s_0$.

\begin{proof}[Proof of Theorem~\ref{thm:main-3}\textup{(i)}]
Write \(d_\phi=D^r\) and \(d_\psi=D^s\) for integers \(D\ge2\) and \(r,s\ge1\). Every composition degree is a power of \(D\). Therefore the interval
\[
T\le\deg u<d_{\max}T=D^{\max\{r,s\}}T
\]
contains at most \(\max\{r,s\}\) possible degree values, independently of \(T\).

Choose \(B>0\) as in Lemma~\ref{lem:sequence-bound}, and put
\[
C_0=d_{\max}(h(P)+B)>0.
\]
For \(u\in W_T\), the calculation in the proof of that lemma gives
\[
h(u(P))\le(\deg u)(h(P)+B)\le C_0T.
\]
Let \(M_T\) be the number of retained intervals in Lemma~\ref{lem:stopping-center-cover}. Each normalized center is of the form \(h(Q)/\deg u\), where \(Q=u(P)\) has height at most \(C_0T\). A fixed \(Q\) contributes at most \(\max\{r,s\}\) such centers. Hence
\[
M_T\le\max\{r,s\}\,N_P(C_0T).
\]
By Lemma~\ref{lem:stopping-center-cover}, the positive-measure hypothesis gives \(M_T\gg T^{s_0}\). Thus \(N_P(C_0T)\gg T^{s_0}\), and putting \(X=C_0T\) yields \(N_P(X)\gg X^{s_0}\). The opposite bound follows from Lemma~\ref{lem:bhz-orbit-upper-bound}.
\end{proof}

For independent degrees, we need a uniform estimate for rotations.
The following is the uniform-phase consequence of the
Erd\H{o}s--Tur\'an inequality; see
Kuipers--Niederreiter~\cite[Chapter~2, Lemma~3.2, pp.~122--123]{KuipersNiederreiter1974}.
\begin{lemma}\label{lem:rotation-discrepancy}
Let \(\alpha\) be irrational and \(x\in\mathbb R\). For a positive integer \(M\), define
\[
D_M(x,\alpha)
=
\sup_{0\le a<b\le1}
\left|
\frac1M
\#\{0\le k<M:\{x+k\alpha\}\in[a,b)\}
-(b-a)
\right|,
\]
where \(\{t\}=t-\lfloor t\rfloor\). Then, for every positive integer \(H\),
\[
D_M(x,\alpha)
\le
C\left(
\frac1H+
\frac1M\sum_{q=1}^{H}\frac1{q\|q\alpha\|}
\right),
\]
where
\[
\|t\|=\min_{m\in\mathbb Z}|t-m|
\]
and \(C>0\) is an absolute constant.
\end{lemma}

\begin{proof}
The Erd\H{o}s--Tur\'an inequality bounds the discrepancy by
\[
C\left(\frac1H+\sum_{q=1}^{H}\frac1q
\left|\frac1M\sum_{k=0}^{M-1}
e^{2\pi i q(x+k\alpha)}\right|\right).
\]
The phase \(e^{2\pi i qx}\) has modulus one, and the geometric sum
has absolute value at most \(C/\|q\alpha\|\). Substitution gives
the assertion, uniformly in \(x\).
\end{proof}

\begin{lemma}\label{lem:diophantine-rotation}
Let \(\alpha\) be irrational, and suppose that
\[
\|q\alpha\|\ge c q^{-\tau}
\qquad(q\ge1)
\]
for some \(c>0\) and \(\tau\ge1\). Then there is a constant \(C>0\), depending only on \(c,\tau\), such that
\[
\left|
\#\{0\le k<M:\{x+k\alpha\}\in[a,b)\}
-M(b-a)
\right|
\le C M^{\tau/(\tau+1)}
\]
for every real \(x\), every integer \(M\ge2\), and \(0\le a<b\le1\).
\end{lemma}
\begin{proof}
By Lemma~\ref{lem:rotation-discrepancy},
\[
D_M(x,\alpha)
\le C_0
\left(
\frac1H+
\frac1M\sum_{q=1}^{H}\frac1{q\|q\alpha\|}
\right),
\]
where \(C_0\) is an absolute constant. Our assumption gives
\[
\sum_{q=1}^{H}\frac1{q\|q\alpha\|}
\le
\frac1c\sum_{q=1}^{H}q^{\tau-1}
\le \frac{H^\tau}{c}.
\]
Choose
\[
H=\left\lfloor M^{1/(\tau+1)}\right\rfloor.
\]
It follows that
\[
M D_M(x,\alpha)
\le C_0\left(\frac MH+\frac{H^\tau}{c}\right)
\le C M^{\tau/(\tau+1)}.
\]
By the definition of \(D_M(x,\alpha)\), this bounds the required counting error for every interval \([a,b)\), uniformly in \(x\).
\end{proof}

The required Diophantine hypothesis follows from the next
fixed-logarithms consequence of
Matveev~\cite[Corollary~2.3, equation~(2.6), p.~1219]{Matveev2000}.
\begin{lemma}\label{lem:matveev-fixed-bases}
Fix positive algebraic numbers \(\beta_1,\ldots,\beta_r\). There are effectively computable constants \(A,C>0\), depending only on these numbers, such that, whenever
\[
\Lambda=b_1\log\beta_1+\cdots+b_r\log\beta_r\ne0,
\qquad b_j\in\mathbb Z,
\]
we have
\[
|\Lambda|\ge A H^{-C},
\qquad
H=\max\{3,|b_1|,\ldots,|b_r|\}.
\]
\end{lemma}

For the next two lemmas, let $k(u)$ denote the number of $\phi$-letters
in a formal word $u$, and put
\[
p=d_\phi^{-s_0},\qquad
\mu=p\log d_\phi+(1-p)\log d_\psi.
\]
Thus $1-p=d_\psi^{-s_0}$. For $A>0$ and $T>1$, write
\[
W_T(A)=\left\{u\in W_T:
\left|k(u)-\frac{p\log T}{\mu}\right|
\le A\sqrt{\log T}\right\}.
\]
Here $D_u=\deg u$, and
\[
I_u(P)=\left[
\frac{h(u(P))-B}{D_u},\frac{h(u(P))+B}{D_u}
\right],
\]
with $B>0$ as in Lemma~\ref{lem:sequence-bound} and $W_T$ as in
Lemma~\ref{lem:stopping-center-cover}.

\begin{lemma}\label{lem:stopping-window-cost}
Let $0<a_0<b_0$, and suppose that
\[
E=\widehat X_{\phi,\psi}(P)\cap[a_0,b_0],
\qquad \mathcal H^{s_0}(E)>0.
\]
Choose $c_0,\delta_0>0$ such that
\[
\mathcal H^{s_0}_{\delta_0}(E)\ge c_0.
\]
For every $\varepsilon>0$, there is $A>0$ such that, for all
sufficiently large $T$,
\[
\sum_{u\in W_T\setminus W_T(A)}
\bigl(\operatorname{diam}I_u(P)\bigr)^{s_0}
\le\varepsilon.
\]
Let $\mathcal I_T(A)$ be the collection of distinct intervals
$I_u(P)$, with $u\in W_T(A)$, that meet $E$. Then, for all
sufficiently large $T$,
\[
\sum_{I\in\mathcal I_T(A)}(\operatorname{diam}I)^{s_0}
\ge c_0-\varepsilon.
\]
In particular, $\varepsilon=c_0/2$ gives a lower bound of $c_0/2$.
\end{lemma}

\begin{proof}
Put $L=\log T$, and assume $L\ge1$. Write
\[
a=\log d_\phi,\qquad b=\log d_\psi,
\qquad R(u)=k(u)-\frac p\mu\log D_u.
\]
In this finite-tree calculation, the empty word has $D_u=1$ and
$k(u)=0$. Appending $\phi$ or $\psi$ gives
\[
R(\phi\circ u)=R(u)+\frac{(1-p)b}{\mu},
\qquad
R(\psi\circ u)=R(u)-\frac{pb}{\mu}.
\]
Since $D_{\phi\circ u}^{-s_0}=pD_u^{-s_0}$ and
$D_{\psi\circ u}^{-s_0}=(1-p)D_u^{-s_0}$, expansion yields
\[
\begin{aligned}
&D_{\phi\circ u}^{-s_0}R(\phi\circ u)^2
+D_{\psi\circ u}^{-s_0}R(\psi\circ u)^2
-D_u^{-s_0}R(u)^2\\
&\hspace{25mm}=\frac{p(1-p)b^2}{\mu^2}D_u^{-s_0}.
\end{aligned}
\]
Sum over the internal words of the finite tree ending at $W_T$.
All internal terms cancel, and the root term is zero. Thus
\[
\sum_{u\in W_T}D_u^{-s_0}R(u)^2
=\frac{p(1-p)b^2}{\mu^2}
\sum_{u\text{ internal}}D_u^{-s_0}.
\]
At depth $j$, the sum of $D_u^{-s_0}$ over all words of that
length is $(p+(1-p))^j=1$. In particular, its sum over the
internal words at that depth is at most $1$. Every stopping word
has length at most $L/\min(a,b)+1$, so
\[
\sum_{u\in W_T}D_u^{-s_0}R(u)^2\le C L.
\]
The proof of Lemma~\ref{lem:stopping-center-cover} gives
\[
\sum_{u\in W_T}D_u^{-s_0}=1,
\qquad 0\le\log D_u-L<\log d_{\max}.
\]
Consequently, for a constant $C_1$ independent of $T$,
\[
\sum_{u\in W_T}D_u^{-s_0}
\left(k(u)-\frac{pL}{\mu}\right)^2\le C_1L.
\]
For the words outside $W_T(A)$, this implies
\[
\sum_{u\in W_T\setminus W_T(A)}D_u^{-s_0}
\le\frac{C_1}{A^2}.
\]
Their original intervals therefore satisfy
\[
\sum_{u\in W_T\setminus W_T(A)}
\bigl(\operatorname{diam}I_u(P)\bigr)^{s_0}
\le\frac{(2B)^{s_0}C_1}{A^2}.
\]
Choose $A$ so that this is at most $\varepsilon$.

By Lemma~\ref{lem:stopping-center-cover}, all the stopping intervals
cover the full completed spectrum, and hence cover $E$. Discarding
retained intervals not meeting $E$ and removing repeated retained
intervals does not affect this cover of $E$ when the excluded
intervals are kept. For sufficiently large $T$, all these intervals
have diameter at most $2B/T\le\delta_0$. Thus
\[
\begin{aligned}
c_0
&\le \sum_{I\in\mathcal I_T(A)}(\operatorname{diam}I)^{s_0}
+\sum_{u\in W_T\setminus W_T(A)}
\bigl(\operatorname{diam}I_u(P)\bigr)^{s_0}\\
&\le \sum_{I\in\mathcal I_T(A)}(\operatorname{diam}I)^{s_0}
+\varepsilon.
\end{aligned}
\]
This proves the second assertion.
\end{proof}

\begin{lemma}\label{lem:height-interval-multiplicity}
Let $E$ be as in Lemma~\ref{lem:stopping-window-cost}. Suppose that
$d_\phi,d_\psi$ are multiplicatively independent and that
\[
\alpha=\frac{\log d_\phi}{\log d_\psi},
\qquad \|q\alpha\|\ge c_\alpha q^{-\tau}\quad(q\ge1)
\]
for some $c_\alpha>0$ and $\tau\ge1$. Put
\[
\eta=\frac{1-s_0}{2(\tau+1)}.
\]
For every fixed $A>0$, there are constants $C,T_0>0$ such that
\[
\#\left\{D_u:
u\in W_T(A),\ h(u(P))=H,\ I_u(P)\cap E\ne\varnothing
\right\}
\le C(\log T)^{1/2-\eta}
\]
for every real $H$ and every $T\ge T_0$. Equivalently, this bounds
the number of distinct retained intervals meeting $E$ that arise
from a fixed ordinary height $H$. The constants are uniform in $H$.
\end{lemma}

\begin{proof}
Write $L=\log T$, $s=s_0$, $a=\log d_\phi$, $b=\log d_\psi$, and
\[
\gamma=\frac1{\tau+1}.
\]
Multiplicative independence implies $0<s<1$. On $\mathbb R/\mathbb Z$,
put
\[
F=\left\{\frac{\log t}{b}\pmod1:t\in E\right\}.
\]
This set is compact. By Lemma~\ref{lem:stopping-center-cover}, the
full completed spectrum has a cover by at most $C\delta^{-s}$
intervals of diameter at most $\delta$, for every sufficiently
small $\delta>0$. Intersect these intervals with $[a_0,b_0]$.
The map $t\mapsto(\log t)/b$ is Lipschitz there, so its images,
reduced modulo $1$, cover $F$ by at most $C\delta^{-s}$ closed arcs
of length at most $C_1\delta$.

For $0<\rho\le\delta$, let
\[
F_\rho=\{y\in\mathbb R/\mathbb Z:
\operatorname{dist}(y,F)\le\rho\}.
\]
Enlarge each of the preceding closed arcs by an open neighborhood
of radius $2\delta$. For sufficiently small $\delta$, the resulting
open arcs have length at most $(C_1+4)\delta$ and cover the closed
set $F_\rho$. Each open arc is contained in at most two half-open
intervals in $[0,1)$, with the same total length, after splitting
at $0$ if necessary.

Apply Lemma~\ref{lem:diophantine-rotation} to these intervals.
The lemma also applies to $-\alpha$, since
$\|q(-\alpha)\|=\|q\alpha\|$. Its uniformity in the initial real
parameter permits any block of $M$ consecutive integers. Thus,
for every real $x$ and every such block,
\[
\#\{k:\{x-k\alpha\}\in F_\rho\}
\le C\delta^{-s}\bigl(M\delta+M^{1-\gamma}\bigr).
\]
Taking $\delta=M^{-\gamma}$ gives, whenever
$\rho\le M^{-\gamma}$ and $M$ is sufficiently large,
\[
\#\{k:\{x-k\alpha\}\in F_\rho\}
\le CM^{1-\gamma(1-s)}.
\]

If $H\le0$, then the relevant set in the statement is empty for
large $T$, because the right endpoint of its interval would be
at most $B/T<a_0$. Suppose $H>0$, and take a word counted in that
set. Choose $t\in I_u(P)\cap E$. Then
\[
\left|\frac H{D_u}-t\right|\le\frac BT.
\]
For large $T$, both quantities inside the logarithm lie in
$[a_0/2,b_0+1]$. In particular,
\[
\left|\frac{\log(H/D_u)}b-\frac{\log t}b\right|
\le\frac{2B}{a_0bT}.
\]
Write $D_u=d_\phi^k d_\psi^\ell$. The logarithmic normalized center
has the exact coordinate
\[
\frac{\log(H/D_u)}b
=\frac{\log H}b-k\alpha-\ell.
\]
Consequently, with $\rho=2B/(a_0bT)$,
\[
\left\{\frac{\log H}b-k\alpha\right\}\in F_\rho.
\]
Thus all these logarithmic normalized centers belong to the open
enlarged arcs constructed above.

The restriction defining $W_T(A)$ places all possible $k$ in a
block of $M$ consecutive integers, where we may take
\[
M=\lceil2A\sqrt L\rceil+3.
\]
There are constants $c_A,C_A>0$ such that
$c_A\sqrt L\le M\le C_A\sqrt L$ for large $T$. Also
$\rho\le M^{-\gamma}$ for large $T$. The preceding rotation
estimate, with $x=(\log H)/b$, therefore bounds the number of
possible $k$ by
\[
CM^{1-\gamma(1-s)}
\le C L^{[1-\gamma(1-s)]/2}
=C L^{1/2-\eta}.
\]
For a fixed $k$, the stopping condition gives
\[
L\le ka+\ell b<L+\log d_{\max}.
\]
It permits at most $\lceil(\log d_{\max})/b\rceil+1$ integers
$\ell$, independently of $T,k,H$. This proves the claimed bound
for distinct degrees. For fixed $H$ and $B>0$, a degree $D$
determines the interval, and different degrees give different
radii. Hence the same bound counts the distinct intervals.
All estimates are uniform in $H$.
\end{proof}

\begin{proof}[Proof of Theorem~\ref{thm:main-3}\textup{(ii)}]
Put \(\alpha=\log d_\phi/\log d_\psi\).
Lemma~\ref{lem:matveev-fixed-bases}, applied to
\(q\log d_\phi-n\log d_\psi\) with \(n\) nearest to \(q\alpha\),
gives constants \(c_\alpha>0\) and \(\tau\ge1\) such that
\[
\|q\alpha\|\ge c_\alpha q^{-\tau}\qquad(q\ge1).
\]
Indeed, multiplicative independence makes the linear form nonzero,
\(|n|\le q\alpha+1/2\), and
\[
\|q\alpha\|
=\frac{|q\log d_\phi-n\log d_\psi|}{\log d_\psi}.
\]
We may take
\[
\eta=\frac{1-s_0}{2(\tau+1)}>0.
\]
Here \(s_0<1\), since independent integer degrees cannot both be
\(2\). Choosing \(c_\alpha,\tau\) by the cited lemma makes
\(\eta\) depend only on the degrees. More generally, any such
\(c_\alpha,\tau\) give this value of \(\eta\), with the lower-bound
constant and the sufficiently-large-\(X\) threshold also allowed
to depend on that choice.

The completed spectrum is nonnegative and is bounded by
Lemma~\ref{lem:sequence-bound}; being closed, it is compact.
Since $s_0>0$, the singleton $\{0\}$ has zero
$\mathcal H^{s_0}$-measure. Countable subadditivity therefore gives
$0<a_0<b_0$ such that
\[
E=\widehat X_{\phi,\psi}(P)\cap[a_0,b_0]
\]
has positive $\mathcal H^{s_0}$-measure. Choose $c_0,\delta_0>0$
with $\mathcal H^{s_0}_{\delta_0}(E)\ge c_0$.

Apply Lemma~\ref{lem:stopping-window-cost} with
$\varepsilon=c_0/2$, and fix the resulting $A$. For each real
$H$, let $m_H(T)$ count the distinct intervals in
$\mathcal I_T(A)$ arising from height $H$. By
Lemma~\ref{lem:height-interval-multiplicity},
\[
m_H(T)\le C(\log T)^{1/2-\eta}
\]
uniformly in $H$ for all sufficiently large $T$.
The finite-composition estimate in the proof of
Lemma~\ref{lem:sequence-bound} gives, for $u\in W_T$,
\[
h(u(P))\le D_u(h(P)+B)\le C_0T,
\qquad C_0=d_{\max}(h(P)+B)>0.
\]
Let
\[
R(T)=\#\{h(Q):Q\in\mathcal O(P),\ h(Q)\le C_0T\}.
\]
Every height contributing to $\mathcal I_T(A)$ is among these
values. Since the stopping intervals have diameter at most $2B/T$,
we obtain
\[
\begin{aligned}
\frac{c_0}{2}
&\le\sum_{I\in\mathcal I_T(A)}(\operatorname{diam}I)^{s_0}\\
&\le\left(\frac{2B}{T}\right)^{s_0}\sum_Hm_H(T)\\
&\le\left(\frac{2B}{T}\right)^{s_0}
R(T)\,C(\log T)^{1/2-\eta}.
\end{aligned}
\]
Here the sum over $H$ is over the finitely many heights arising
from the retained stopping words. Rearranging and putting
$X=C_0T$ proves the lower bound for distinct ordinary height
values, and hence the lower bound for $N_P(X)$.
The upper bound follows from
\cite[Theorem~1.5]{BellHindesZhongApplications}.
\end{proof}

In the cubic family we can count one full layer of the
orbit and obtain matching bounds.

\begin{corollary}\label{cor:cubic-orbit-count}
Retain the hypotheses and the point \(x\) of
Proposition~\ref{prop:rational-cubic-spectrum}. Put
\[
N=\max(a,b),\quad r=\frac1{2N},\quad
\varepsilon_{\mathrm{cub}}=\log\bigl(1+(N-3)r^2+2r^3\bigr),\quad
\rho=\frac{\log2}{\log3},
\]
and
\[
c_-=h(x)-\tfrac12\log N>0,
\qquad c_+=h(x)+\varepsilon_{\mathrm{cub}}/2.
\]
Then, for \(X\ge c_+\),
\[
\frac12c_+^{-\rho}X^\rho
\le N_x(X)\le2c_-^{-\rho}X^\rho.
\]
In particular,
\[
N_x(X)\asymp X^\rho.
\]
\end{corollary}

\begin{proof}
The coding separation in
\eqref{eq:cubic-coding-separation} implies that distinct
length-\(n\) words give distinct points at \(x\).
Indeed, if \(u(x)=u'(x)\) and \(|u|=|u'|=n\),
appending the same infinite tail \(\omega\) gives
\[
\widehat h_{\omega\circ u}(x)
=3^{-n}\widehat h_\omega(u(x))
=3^{-n}\widehat h_\omega(u'(x))
=\widehat h_{\omega\circ u'}(x),
\]
contrary to the separation of their distinct sequences.
Thus exactly \(2^n\) distinct points occur at depth \(n\).

Since \(L\ge2N\),
\(h(x)=\log(1+LT)>\tfrac12\log N\), so \(c_->0\).
Iterating \eqref{eq:cubic-height-recurrence} gives
\[
c_-3^n\le h(u(x))\le c_+3^n
\qquad(|u|=n).
\]
For \(X\ge c_+\), choose
\(n=\lfloor\log_3(X/c_+)\rfloor\).
The \(2^n\) distinct points at this depth all have height
at most \(X\), so
\[
N_x(X)\ge2^n\ge\tfrac12(X/c_+)^\rho.
\]
Conversely, a representation of a point of height at most
\(X\) can only have length at most
\(\lfloor\log_3(X/c_-)\rfloor\). Counting all words
of these lengths, including the empty prefix, gives
\[
N_x(X)
\le\sum_{n=0}^{\lfloor\log_3(X/c_-)\rfloor}2^n
\le2(X/c_-)^\rho.
\]
Coincidences between different depths do not affect either
bound. The assertion concerns two-sided bounds, not an
asymptotic equivalent with a limiting leading constant.
\end{proof}

\section{Discussion}\label{sec:discussion}

The upper bound depends only on the degrees, but the spectrum
dimension also depends on the starting point. Consider
\(\phi(t)=t^2\) and \(\psi(t)=\alpha t^2\), where
\(\alpha=(3+4i)/5\) and \(h(\alpha)=\tfrac12\log5\).
Every composition of length \(n\) is \(u(t)=\alpha^m t^{2^n}\),
with every \(0\le m\le2^n-1\) occurring. For \(z\ne0,\infty\),
conjugating this monomial to a power map gives
\(\widehat h_u(z)=F_z(m/(2^n-1))\), where
\[
F_z(s)=\frac1{[K:\mathbb Q]}\sum_{v\in M_K}n_v
\max\{0,\log|z|_v+s\log|\alpha|_v\},\qquad 0\le s\le1,
\]
and \(K\) contains \(\alpha,z\). Only finitely many places contribute.
The grid parameters are dense in \([0,1]\), and \(F_z\) is
continuous and piecewise affine. Thus the completed spectrum is
\(F_z([0,1])\), an interval or a singleton. This pair admits no
dimension strictly between zero and one.

More explicitly,
\[
\begin{aligned}
\widehat X_{\phi,\psi}(1)&=[0,\tfrac12\log5],\\
\widehat X_{\phi,\psi}(2)&=[\log2,\log2+\tfrac12\log5],\\
\widehat X_{\phi,\psi}(1/5)&=\{\log5\}.
\end{aligned}
\]
At \(2\) and \(1/5\), both individual canonical heights are
positive, so both individual orbits are infinite. Their spectrum
dimensions are nevertheless different. The points \(0\) and infinity
give the singleton \(\{0\}\). Intermediate dimension occurs for a
different pair: the cubic example \(a=3,b=4\) in
Proposition~\ref{prop:rational-cubic-spectrum} has dimension
\(\log2/\log3\) at \(1/73\).

For the fixed cubic pair \(\theta_3,\theta_4\), the general bound is
\(\rho=\log2/\log3\). The common fixed points \(0,1\) give dimension
zero, and \(1/73\) gives dimension \(\rho\). These results do not
determine whether a value strictly between zero and \(\rho\) occurs.
Because the starting points are algebraic, the set of attained
dimensions is countable; it cannot be the whole real interval
\([0,\rho]\). One may instead ask which other values occur and
whether this countable set can be dense in that interval.

For the quadratic pair, \(s_0=1\) and
\(d_\phi^{-1}+d_\psi^{-1}=1\), the boundary of the strict
inequality in Corollary~\ref{cor:intermediate-dimension}.
At \(1\), the first-letter height intervals
\([0,h(\alpha)/2]\) and \([h(\alpha)/2,h(\alpha)]\) meet.
Passing to \(\phi^j,\psi^j\) gives containing intervals
\([0,2^{-j}h(\alpha)]\) and
\([(1-2^{-j})h(\alpha),h(\alpha)]\), separated for \(j\ge2\).
This illustrates the use of iterates in the height ping-pong
argument of Bell--Huang--Peng--Tucker
\cite[Proposition~3.1 and Lemma~3.2]{BHPT2024}.

Identical canonical-height functions give another source of
singleton spectra; see also
Kawaguchi--Silverman~\cite{KawaguchiSilverman2007Identical}.
If \(\widehat h_\phi=\widehat h_\psi=H\), then
\(H(u(P))=(\deg u)H(P)\) for every nonempty word. Since \(H-h\) is bounded,
taking the defining limit for \(u\) gives \(\widehat h_u=H\).
Consequently \(\widehat X_{\phi,\psi}(P)=\{H(P)\}\) for every \(P\).

Proposition~\ref{prop:orbit-counting-lower} turns positive spectrum dimension into a
lower orbit-counting exponent. The quadratic example is explicit:
\(\mathcal O(1)=\{\alpha^m:m\ge0\}\), since the exponent maps are
\(m\mapsto2m\) and \(m\mapsto2m+1\). As
\(h(\alpha^m)=mh(\alpha)\),
\[
N_1(X)=\left\lfloor\frac{X}{h(\alpha)}\right\rfloor+1
\qquad(X\ge0).
\]
Corollary~\ref{cor:cubic-orbit-count} gives another example with matching upper and lower
bounds, \(N_x(X)\asymp X^{\log2/\log3}\).

For general pairs, which values in \([0,s_0]\) occur among points
with both individual orbits infinite? How do local contributions
describe the zero- and positive-dimensional loci? Further counting
questions concern the critical exponent without a logarithmic
loss for independent degrees, and asymptotic equivalents beyond
two-sided bounds. It would also be useful to determine which of
these conclusions extend to more than two generators.

\clearpage
\bibliographystyle{plain}
\begingroup
\small\raggedright
\bibliography{references}
\endgroup

\end{document}